\documentclass[review]{elsarticle}

\usepackage{lineno,hyperref}
\usepackage{graphicx,hyperref}
\usepackage{amsmath}
\usepackage{amsthm}
\usepackage{bm} 
\usepackage{amssymb}

\usepackage{subcaption}
\usepackage{pdfpages}

\newtheorem{thm}{Theorem}[section]
\newtheorem{rem}{Remark}[section]
\newtheorem{cor}{Corollary}[section]

\modulolinenumbers[5]

\journal{Journal of \LaTeX\ Templates}

\biboptions{authoryear}

\begin{document}

\begin{frontmatter}

\title{A Bayes interpretation of stacking for $\cal{M}$-complete and $\cal{M}$-open settings }


 \author{Tri Le \fnref{fn1}}
 \ead{tle20@unl.edu}

 \author{Bertrand Clarke \fnref{fn1}}
 \ead{bclarke3@unl.edu}

  \fntext[fn1]{Department of Statistics, University of Nebraska-Lincoln.  }

 \begin{abstract}
  In ${\cal{M}}$-open problems where no true model can be conceptualized, it is common to 
back off from modeling and merely seek good prediction. Even in ${\cal{M}}$-complete problems,
taking a predictive approach can be very useful.
Stacking is a model averaging procedure that gives a composite predictor by combining individual predictors from a list of models
using weights that optimize a cross-validation criterion.  We show that the
stacking weights also asymptotically minimize a posterior expected loss.   Hence we formally provide a Bayesian
justification for cross-validation.  Often the weights
are constrained to be positive and sum to one.  For greater generality, we omit the positivity constraint
and relax the `sum to one' constraint.

A key question is `What predictors should be in the average?'
We first verify that the stacking error depends only on the span of the models.
Then we propose using bootstrap samples from the data to  generate empirical 
basis elements that can be used to form models. We use this in two computed examples
to give stacking predictors that are (i) data driven, (ii) optimal
with respect to the number of component predictors, and (iii) optimal with respect
to the weight each predictor gets.
 
\end{abstract}


\end{frontmatter}


\section{Introduction}
Stacking is a model averaging procedure for generating predictions first introduced by 
\citet{wolpert}.  The basic idea is that if $J$ candidate signal plus noise models of the form
$Y = f_j(x) + \epsilon$ for $j= 1, \ldots ,  J$ are available then
they can be usefully combined to give the predictor
$$
\hat{Y}(x) = \sum_{j=1}^J \hat{w}_j \hat{f}_j(x) ,
$$
where $\hat{f}_j$ is an estimate of $f_j$.  Usually, $f_j(x)  = f_j(x, \beta_j)$ 
so $\hat{f}_j(x) = f_j( x , \hat{\beta}_j)$ where $\hat {\beta}_j$ is an estimate of $\beta_j$.
The weights $\hat{w} =  (\hat{w}_1, \ldots , \hat{w}_J)$ satisfy
\begin{eqnarray}
\hat{w} = \arg \min_{w} \sum_{i=1}^n\left( y_i - \sum_{j=1}^J w_j \hat{f}_{j ,-i}(x_i) \right)^2
\label{stackcrit}
\end{eqnarray}
where $\hat{f}_{j,-i}$ is the estimate of $f_j$ using the $n-1$ of the $n$ data points by dropping the $i$-th
one, i.e., $(x_1, y_1), \ldots, (x_{i-1}, y_{i-1}), (x_{i+1}, y_{i+1}), \ldots , (x_n, y_n)$.  The $Y_i$'s are
assumed independent and the $x_i$'s are deterministic design points.  Often the $w_j$'s are
assumed to be non-negative and sum to one.
The properties of stacking as a predictor have been explored in numerous contexts such as regression
\citet{breiman}, \citet{clarke}, \citet{sill}, classification and distance learning \citet{ting}, \citet{ozay},  density estimation \citet{smyth},
and estimating bagging's error rate \citet{rokach}, \citet{macready}.

These earlier contributions treated stacking as a frequentist procedure.  However, more recently, \citet{clyde}
brought stacking into the Bayesian paradigm.   They recalled the tripartite partition of statistical 
problems into three classes namely ${\cal{M}}$-closed, ${\cal{M}}$-complete, ${\cal{M}}$-open, see \citet{bernardo},
and suggested that outside the ${\cal{M}}$-closed setting the posterior risk could be approximated by a cross-validation error
(for the same loss function).  Hence the action minimizing the posterior risk could be approximated by 
the stacking predictor that minimizes \eqref{stackcrit}.  More precisely, given models $M_j$ for $j=1, \ldots , J$,
 a loss function $\ell$, a vector of responses $\bm{Y}=\bm{y} = (y_1, \ldots, y_n)$, and an element $a(\bm{y})$ in the
action appropriate for a collection of models, say ${\cal{M}}$, \citet{clyde} used
\begin{eqnarray}
\int  \ell(y_{n+1} , a(\bm{y})) p(y_{n+1} \mid \bm{y}) {\rm d} y_{n+1}  \approx \frac{1}{n} \sum_{i=1}^n \ell( y_i, a(\bm{y}_{-i}))
\label{Clydeapprox}
\end{eqnarray}
in an ${\cal{M}}$-open context,
where $y_{n+1}$ represents a future outcome at a future design point $x_{n+1}$, $\bm{y}_{-i}$ is the data vector
$\bm{y}$ with the $i$-th entry deleted, and $p( \cdot \mid \bm{y})$ is the predictive distribution for $Y_{n+1}$.
Here and elsewhere, the design points $x_1, \ldots , x_{n+1}$ are suppressed in the notation unless consideration of
them is essential for a step in a proof.
Hence, \citet{clyde} observed that minimizing the left hand side of \eqref{Clydeapprox} over $a(\bm{y})$ and the right hand side over $a(\bm{y}_{-i})$
leads to two actions that are asymptotically identical.  Otherwise put, the stacking predictor is the asymptotic
Bayes action for ${\cal{M}}$-complete problems. It is not the Bayes action in the  ${\cal{M}}$-open case because the mode of convergence is
undefined. Nevertheless,   \citet{clyde} used  \eqref{Clydeapprox} in an  ${\cal{M}}$-open context to good effect.  It should be noted that \eqref{Clydeapprox}
seems to have been initially conjectured in \citet{bernardo} and a non-cross-validatory version of
\eqref{Clydeapprox} for individual models $M_j$, namely
$$
E_{Y_{n+1} \mid \bm{Y}, M_j} \ell( Y_{n+1}, a_{M_j}(\bm{Y})) - \frac{1}{n} \sum_{i=1}^n \ell( Y_i, a_{M_j}(\bm{Y})) \stackrel{P}{\rightarrow}  0,
$$
is established in \citet{walker} where $a_{M_j}(\bm{Y})$ is in the action space associated with $M_j$.

Aside from the applications of these results to the stacking predictor, the results -- if proved formally as below --
establish that leave-one-out cross validation is asymptotically a Bayes optimal procedure under some conditions.
It can be verified that the proofs below extend to leave-$k$-out cross-validation as well.  That is, our results provide
a Bayesian justification for using cross-validation as a way to choose a model from
which to generate predictions outside
of ${\cal{M}}$-closed problems.

For the sake of completeness, we recall that \citet{bernardo} define ${\cal{M}}$-closed problems
as those for which a true model can be identified and written down but is one amongst finitely
many models from which an analyst has to choose.  By contrast, ${\cal{M}}$-complete problems
are those in which a true model (sometimes called a belief model) exists but is
inaccessible in the sense that even though it can be conceptualized
it cannot be written down or at least cannot be used
directly.  Effectively this means that other surrogate models must be identified
and used for inferential purposes.  ${\cal{M}}$-open problems according to \citet{bernardo}
are those problems where a true model exists but cannot be specified at all.  

Here however, we make a stronger distinction between ${\cal{M}}$-complete and ${\cal{M}}$-open problems
by taking the view that in the ${\cal{M}}$-open
case no true model can even be conceptualized.  Hence it is inappropriate to assume the existence of a true model.
We prefer this stronger distinction because it ensures that
${\cal{M}}$-complete and ${\cal{M}}$-open are disjoint classes.  In both ${\cal{M}}$-complete
and ${\cal{M}}$-open classes the status of the prior is unclear because none of the models under consideration
are taken is true.   However, a weighting function ostensibly indistinguishable from a prior can be regarded as a sort of 
pseudo-belief in the sense that 
it is the weight one would pre-experimentally assign to the model  if it were an action for predicting the outcomes of 
a data generator.     More generally, the weights can only be interpreted as an index for a class of actions, provided
the weighted combination of predictors from the $J$ models is regarded as an action.

Although the $w_j$'s are often assumed to be positive and sum to one e.g., \citet{clyde}, \citet{breiman} only assumed the
weights were positive and some remarks in \citet{clyde} consider the case that the weights only satisfy a `sum to one'
constraint thereby permitting negative weights.  We can see the effect of the sum to one constraint in a simple
example.  Following \citet{clydeslide}, consider the two models $M_1: Y = x_1\beta_1 + \epsilon$ and $M_2: Y = x_2\beta_2 + \epsilon$ where
the explanatory variables are orthogonal i.e., $x_1' x_2 = 0$.    As shown in the Appendix at the end of the paper, if we stack these two 
models with the sum to one constraint we get $\hat{w}_1 = \hat{w}_2 = 1/2$.  That is,
 predictions are generated from $Y_W = (1/2) x_1 \hat{\beta}_1 + (1/2) x_2 \hat{\beta}_2$ where the $\hat{\beta}_k$
are found from model $M_k$ for $k=1, 2$.  
On the other hand, if we stack $M_1$ and $M_2$ without the sum to one constraint but with, say, a sum to two
constraint we get $Y_{WO} = x_1\hat{\beta}_1 + x_2\hat{\beta}_2+ \epsilon$, i.e.,  $\hat{w}_1 = \hat{w}_2 = 1$.
Obviously, $Y_{WO} = 2 Y_W$ so $Y_W$ is half the size it should be.  This extends to three or more models and
shows that the sum to one constraint can be too restrictive.  In addition, permitting $w_j$'s to be negative increases
the range of the stacking predictors and can only result in better predictions.  Consequently in most of our results below
we do not impose either the sum to one constraint or the non-negativity constraint.

There three main contributions of this paper are (1) a formal proof that \eqref{Clydeapprox} holds for several
loss functions in  ${\cal{M}}$-complete settings, (2) explicit formulae for the stacking weights for various choices of constraints on the $w_j$'s,
and (3) a way to choose optimal basis expansions to stack so as to clarify the suggestion in \citet{breiman} that the
models be chosen as different from each other as possible. In our examples, we choose two data generators, one  ${\cal{M}}$-complete and one
${\cal{M}}$-open, to see how stacking performs.

The structure of this paper is as follows. In Section \ref{postexputil} we present the formal proof of using cross-validation to approximate posterior risk and hence
derive stacking as an approximation to the Bayes action.
In Section \ref{derivation}, we use the approximation to the posterior expected risk to derive stacking weights under several sets of constraints
on the weights, observing that relaxing the non-negativity and the sum to one constraints improves prediction.
In Section \ref{basisselection}, we show how to get optimal data-driven basis expansions to stack. 
These bases should be different from each other in the sense of being independent;
orthogonality does not seem to be helpful. 
In Section \ref{computedexamples}, we give two real data examples where the ${\cal{M}}$-complete or ${\cal{M}}$-open assumption is reasonable.
We use them to show the effect of the sum of the coefficients and to suggest desirable properties of  basis element generation.
Some concluding remarks are made in Section \ref{Discussion}. 
\section{Approximating posterior risk}
\label{postexputil}
Let ${\cal{M}} = \{ M_1, \ldots ,  M_J\}$ be a class of models and $\bm{y} = (y_1,\cdots,y_n)$ be the vector of outcomes of $\bm{Y} = (Y_1, \ldots , Y_n)$, where the $Y_i$'s are 
independently distributed with  probability density function (pdf) $p_j(y\mid {\theta}_j)$ for $j=1 \ldots , J$ 
equipped with a prior $w_j({\theta}_j)$. Consider a loss function $\ell: \mathbb{R} \times {\cal{A}} \rightarrow \mathbb{R}$
where ${\cal{A}}$ is the action space of a predictive decision problem.  In this setting $\ell(y_{n+1},a(\bm{y}))$ is the cost of taking action $a(\bm{y})$,
where $y_{n+1}$ is a future observation.    Although the language of utility functions is more common in our context, we prefer the language of loss functions 
because it is more suggestive of decision theory. The posterior risk under model $j$ is
\begin{equation}
\int \ell( y_{n+1},a( \bm{y} ) ) p_j(y_{n+1} \mid \bm{y} ) d y_{n+1},
\end{equation}
where $p_j( \cdot \mid \bm{y} )$ is the predictive density from model $j$.  Given a set of convex weights $\pi(j)$ for use over the models, the 
overall posterior risk is
\begin{equation}
 \int \ell( y_{n+1},a( \bm{y} ) ) p(y_{n+1} \mid \bm{y} ) d y_{n+1}
\end{equation}
where $p(y_{n+1} \mid \bm{y})$ is the predictive density marginalizing out over $j$ as well as the $\theta_j$'s.

The relationship between the notation in \eqref{stackcrit} and the above is that
if $Y = f_j(x) + \epsilon$ we can write $f_j$ in a generic parametric form $f_j(x) = f_j(x, \beta_j)$ so that
$Y \sim p_j(y \mid {\theta}_j)$ means $Y \sim p_j(y \mid x, {\theta}_j)$ where ${\theta}_j$ is
the concatenation of $\beta_j$ and the parameters in the distribution of $\epsilon$.  
We also assume without further comment that (i) the explanatory variable $x$ and the parameter $\theta_j$ are of fixed dimension,
and, for simplicity of notation, (ii) $(x, \theta_j) \in K_1 \times K_2$
where $K_1$ is a compact set in the space of explanatory variables and $K_2$ is a compact set.  Strictly speaking,
$K_2$ depends on $j$, but we assume that a single $K_2$ can be found and used for all $j$.  This latter regularity condition can be relaxed 
at the cost of more notation.  As a separate issue, because the Bayes predictors require integration over $\theta$,
the compactness of $K_2$ is only needed for the frequentist results.

In the results below we establish six versions of \eqref{Clydeapprox} using three different loss functions (squared error,
absolute error, and logarithmic loss -- also sometimes called a logarithmic scoring rule) and two different classes of
predictor (Bayes and plug-in).   Bayes predictors are of the form
$E_j( Y_{n+1} \mid \bm{Y})$ and plug-in predictors are of the form $E_{\hat{\theta}_j} Y_{n+1}$ where $\hat{\theta}_j$ is an estimator
of the true value of $\theta_j$ using $\bm{Y}$. To an extent the proofs of these results are similar:  All of them use multiple
steps of the form `add and subtract the right extra terms, apply the triangle inequality, and bound
the result term-by-term', and conclude by invoking a uniform integrability condition.    One difference is that the results for the Bayes predictors 
invoke a martingale convergence theorem whereas plug-in predictors
add an extra step based on the consistency of the $\hat{\theta}_j$'s.  

We begin by giving conditions under which we can state and prove \eqref{Clydeapprox} for squared error and 
Bayes predictors; this provides formal justification for the methodology in \citet{clyde} in ${\cal{M}}$-complete problems.

\begin{thm}\label{thm1}
Let $\ell(z, a) = (z-a)^2$ denote squared error loss.  Assume
\begin{description} 
\item (i) For any $j=1, \ldots , J$ and any pre-assigned $\epsilon>0$,  
	$$
	E_j(Y^{4+\epsilon}) = \int \int y^{4+\epsilon} p_j(y\mid x, {\theta}_j) w_j({\theta}_j) d\theta_j d y < \infty, 
	$$ 
	and $E_j (Y^{4+\epsilon})$ is continuous for $x \in K_1$,
\item (ii) For each $j = 1, \ldots , J$, the conditional densities $p_j( y \mid x, \theta_j)$ are equicontinuous for $x \in K_1$ for
each $y$ and $\theta_j \in K_2$, and,
\item (iii) For each $j\ = 1, \ldots , J$, the Bayes predictor
	$$
	\hat{Y}_{j} = E_j(Y_{n+1} \mid\bm{Y}) = \int\int y_{n+1}  p_j(y_{n+1} \mid x_{n+1}, {\theta}_j) w_j({\theta}_j\mid\bm{Y}) d{\theta}_j d y_{n+1}
	$$
	is used to generate predictions at the $n+1$ step.
\end{description}
Then, for any action $a(\bm{Y})= \sum_{j=1}^J w_j \hat{Y}_j$, $w_j \in \mathbb{R}$ for all $j$, we have
\begin{equation}
\nonumber
	\int \ell(y_{n+1},a(\bm{Y})) p(y_{n+1}\mid\bm{Y}) d{y}_{n+1} - \frac{1}{n} \sum_{i=1}^n \ell(Y_i,a(\bm{Y}_{-i})) \stackrel{L^2}{\rightarrow} 0 \mbox{ as } n\rightarrow\infty.
\end{equation}

\end{thm}
\begin{proof}
Fix a countable sequence $x_1, x_2, \ldots \in K_1$.
To establish the theorem, it is enough to show that as $n \rightarrow \infty$,
\begin{equation*}
	E_{\bm{Y}} \left[ 	\int \ell(y_{n+1},a(\bm{Y})) p(y_{n+1}\mid\bm{Y}) d{y}_{n+1} - \frac{1}{n} \sum_{i=1}^n \ell(Y_i,a(\bm{Y}_{-i}))\right] ^2 \rightarrow 0 \label{2.4.2.13}.
\end{equation*}
This can be done by adding and subtracting $(1/n) \sum_{i=1}^n \ell(Y_i,a(\bm{Y}))$ to see 
\begin{eqnarray} 
\label{2.4.2.14}
\begin{aligned}
	 & 2 E_{\bm{Y}}  \left[  	E_{Y_{n+1} \mid\bm{Y}} \ell(Y_{n+1},a(\bm{Y})) - \frac{1}{n} \sum_{i=1}^n \ell(Y_i,a(\bm{Y}))\right]^2 \\
	 & +2 E_{\bm{Y}} \left[  \frac{1}{n} \sum_{i=1}^n \ell(Y_i,a(\bm{Y})) - \frac{1}{n} \sum_{i=1}^n \ell(Y_i,a(\bm{Y}_{-i}))\right]^2   
\end{aligned}	 
\end{eqnarray}
bounds the left hand side of the last expression from above and goes to zero.

{\it Step 1:  The first term in (\ref{2.4.2.14}) has limit zero.}
To see this, recall that \citet{walker} establish 
\begin{eqnarray*}
	E_{Y_{n+1}\mid\bm{Y}, M_j} \ell(Y_{n+1},a_{M_j}(\bm{Y})) - \frac{1}{n} \sum_{i=1}^n \ell(Y_i,a_{M_j}(\bm{Y})) \stackrel{P}{\rightarrow} 0 \mbox{ as } n\rightarrow\infty,
\end{eqnarray*}   
for any $j$, where $E_{Y_{n+1}\mid\bm{Y}, M_j} $ denotes the conditional expectation with respect to $(Y_{n+1} \mid \bm{Y})$ within model $M_j$; 
see also \citet{clyde} for a proof valid under our hypotheses.  Summing over $j = 1, \ldots , J$ and recognizing that
$a_{M_j}$ is a function of $\bm{Y}$, and hence can be regarded as an element $a(\bm{Y})$ of the action space of ${\cal{M}}$, give
\begin{eqnarray*}
	E_{Y_{n+1}\mid\bm{Y} } \ell(Y_{n+1},a(\bm{Y})) - \frac{1}{n} \sum_{i=1}^n \ell(Y_i,a(\bm{Y})) \stackrel{P}{\rightarrow} 0 \mbox{ as } n\rightarrow\infty ,
\end{eqnarray*}   
and hence
\begin{eqnarray}
	\quad\left[ E_{Y_{n+1}\mid\bm{Y}} \ell(Y_{n+1},a(\bm{Y})) - \frac{1}{n} \sum_{i=1}^n \ell(Y_i,a(\bm{Y}))\right] ^2 \stackrel{P}{\rightarrow} 0 \mbox{ as } n\rightarrow\infty. \label{2.4.2.25}
\end{eqnarray}   
Since $\ell$ is squared error and $a(\bm{Y})=\sum_{j=1}^J w_j \hat{Y}_j$, the left hand side of (\ref{2.4.2.25}) is
\begin{eqnarray} 
	&&  \left\lbrace   \left[ E(Y_{n+1}^{2}\mid\bm{Y})-\frac{1}{n} \sum_{i=1}^n Y_i^2 \right] \right.\nonumber\\
	&&\quad\quad\quad\quad\left.+ 2\left(\sum_{j=1}^J w_j \hat{Y}_j \right) \left[  \frac{1}{n} \sum_{i=1}^n Y_i - E(Y_{n+1}\mid\bm{Y})    \right]          \right\rbrace ^2 \label{2.4.2.26}\nonumber\\
	\quad& = &\left\lbrace    E(Y_{n+1}^{2}\mid\bm{Y})+\left(-\frac{1}{n} \sum_{i=1}^n Y_i^2 \right)\right.\nonumber\\
	\quad&&\left.+ \left[ 2\left(\sum_{j=1}^J w_j E_j(Y_{n+1}\mid\bm{Y}) \right) \left( \frac{1}{n} \sum_{i=1}^n Y_i - E(Y_{n+1}\mid\bm{Y})    \right)\right] \right\rbrace ^2 ,
	\label{2.4.2.27}
\end{eqnarray} 
since $\hat{Y}_{j} = E_j(Y_{n+1}\mid\bm{Y})$ by Assumption (iii).
Now, using $(a+b+c)^2 \leq 3(a^2 + b^2 + c^2)$ on the right hand side of \eqref{2.4.2.27}, the left hand side of \eqref{2.4.2.25} is bounded by
\begin{eqnarray} 
	&&  3E^2(Y_{n+1}^{2}\mid\bm{Y}) + 3 \left( \frac{1}{n} \sum_{i=1}^n Y_i^2\right) ^2 \nonumber \\
	&& + 12 \left[ \sum_{j=1}^J w_j E_j(Y_{n+1}\mid\bm{Y})\right] ^2 \left[\frac{1}{n} \sum_{i=1}^n Y_i - E(Y_{n+1}\mid\bm{Y})  \right] ^2 \nonumber\\
	\quad\quad&& \le 3E(Y_{n+1}^{4}\mid\bm{Y})+ 3 \left( \frac{1}{n} \sum_{i=1}^n Y_i^2\right) ^2 \nonumber\\
	&&\quad+ 24 \left( \sum_{j=1}^J w_j^2\right)  \left[ \sum_{j=1}^J  E_j^2(Y_{n+1}\mid\bm{Y})\right] \nonumber\\
	&&\quad\quad\times\left[ \left(\frac{1}{n} \sum_{i=1}^n Y_i \right) ^2 + E^2(Y_{n+1}\mid\bm{Y}) \right]. \label{2.4.2.271}
\end{eqnarray}
Using $ab\le (a^2+b^2)/2$ in the third term of (\ref{2.4.2.271}) gives the new bound
\begin{eqnarray} 
\begin{aligned}
	&\quad  3E(Y_{n+1}^{4}\mid\bm{Y})+ 3 \left( \frac{1}{n} \sum_{i=1}^n Y_i^2\right) ^2 \\
	&	+ 12 \left( \sum_{j=1}^J w_j^2\right)  \left[ \sum_{j=1}^J  E_j^2(Y_{n+1}\mid\bm{Y})\right]^2 \\
		& +  12  \left( \sum_{j=1}^J w_j^2\right) \left[ \left(\frac{1}{n} \sum_{i=1}^n Y_i \right) ^2 + E^2(Y_{n+1}\mid\bm{Y}) \right]^2.  \label{2.4.2.272}
\end{aligned}
\end{eqnarray}
Using $(\sum_i a_i)^2 \leq n\sum_i a_i^2$ in the third and fourth terms of \eqref{2.4.2.272} and then applying Cauchy-Schwarz 
to the results gives the upper bound
\begin{eqnarray}  
\begin{aligned}
		&\quad 3E(Y_{n+1}^{4}\mid\bm{Y})+ 3 \left( \frac{1}{n} \sum_{i=1}^n Y_i^2\right) ^2 \\
		&\quad	+ 12 J \left( \sum_{j=1}^J w_j^2\right)  \left( \sum_{j=1}^J  E_j(Y_{n+1}^{4}\mid\bm{Y})\right)  \\
			&\quad + 24 \left( \sum_{j=1}^J w_j^2\right) \left( \frac{1}{n} \sum_{i=1}^n Y_i^2\right) ^2  
			+ 24 \left( \sum_{j=1}^J w_j^2\right)  E(Y_{n+1}^{4}\mid\bm{Y}) . \label{2.4.2.28}
	\end{aligned}
\end{eqnarray}

Next, we show that each term on the right in \eqref{2.4.2.28} is uniformly integrable.  First observe that, under
Assumption (i), $E(Y_{n+1}^{4}\mid\bm{Y})$ and $E_j(Y_{n+1}^{4}\mid\bm{Y})$ are uniformly integrable, see \citet{billingsley}, p. 498,
so terms one, three, and five in \eqref{2.4.2.28} are uniformly integrable.
Since terms two and four are nearly the same, it is enough to show term two is uniformly integrable.
Begin by noting that Jensen's inequality gives
\begin{eqnarray}   
	\left( \sum_{i=1}^n a_i\right) ^{1+\epsilon} \le n^\epsilon \left( \sum_{i=1}^n a_i^{1+\epsilon}\right), \label{21}
\end{eqnarray}
for any $\epsilon >0$ and $a_i \ge 0, i=1,\cdots,n$.  (Set $\varphi(x) = x^{1+\epsilon}$.)  Now, by Cauchy-Schwarz and \eqref{21}
\begin{eqnarray*} 
	\sup_n E \left[ \left( \frac{1}{n} \sum_{i=1}^n Y_i^2\right) ^2 \right] ^{1+\epsilon}
	&\le& \sup_n E \left( \frac{1}{n} \sum_{i=1}^n Y_i^4 \right) ^{1+\epsilon} \\
	&\le& \sup_n \frac{1}{n^{1+\epsilon}} E \left[ n^\epsilon \left( \sum_{i=1}^n Y_i^{4(1+\epsilon)} \right) \right]  \\
	&\le& \sup_x E \left( Y^{4(1+\epsilon)}\right)  < \infty,
\end{eqnarray*} 
by Assumption (i) where $Y$ denotes any $Y_i$ as a function of $x$.  Thus term two in \eqref{2.4.2.28} is uniformly integrable.

Since all the terms on the right hand side of (\ref{2.4.2.28}) are uniformly integrable,
\begin{eqnarray*}  
	&&\left[  	E_{Y_{n+1}\mid\bm{Y}} \ell(Y_{n+1},a(\bm{Y})) - \frac{1}{n} \sum_{i=1}^n \ell(Y_i,a(\bm{Y}))\right]^2 
\end{eqnarray*} 
is uniformly integrable and by (\ref{2.4.2.25}) has limit zero in probability. Thus
\begin{eqnarray}  
	E_{\bm{Y}}  \left[  	E_{Y_{n+1}\mid\bm{Y}} \ell(Y_{n+1},a(\bm{Y})) - \frac{1}{n} \sum_{i=1}^n \ell(Y_i,a(\bm{Y}))\right]^2 \rightarrow 0
\end{eqnarray} 
as $n\rightarrow\infty$ by the Theorem 25.12 in \citet{billingsley}. So the first term in \eqref{2.4.2.14} also goes to $0$ as $n\rightarrow\infty$.

{\it Step 2: Obtain a bound on the second term in \eqref{2.4.2.14}.}  Using $(\sum_i a_i)^2 \leq n\sum_i a_i^2$ it is seen that
the second term in \eqref{2.4.2.14} is bounded by twice
\begin{eqnarray}  
\label{2.4.2.15}
\begin{aligned}
	 \frac{1}{n} E_{\bm{Y}} \sum_{i=1}^n \left[  \ell(Y_i,a(\bm{Y})) -  \ell(Y_i,a(\bm{Y}_{-i}))\right]^2 .
\end{aligned}
\end{eqnarray} 
The terms in \eqref{2.4.2.15} depend on $x_1, \ldots , x_n \in K$ and the $i$-th term depends
on $x_i$ differently from how it depends on the $x_{i^\prime}$'s for $i^\prime \neq i$.  However, the sum
is symmetric in the $x_i$'s.

Recalling $\ell(z,a)= (z-a)^2$ and $a(\bm{Y})=\sum_{j=1}^J w_j \hat{Y}_j$, where $w_j$ is a coefficient assigned to model $M_j$, 
the  $i$-th term in \eqref{2.4.2.15} is
\begin{eqnarray}  \label{2.4.2-15}
\begin{aligned}
	& E_{\bm{Y}} \Bigg\{  \Bigg[ \bigg( \sum_{j=1}^J w_j \hat{Y}_j\bigg) ^2 - \bigg( \sum_{j=1}^J w_j \hat{Y}_{j,-i}\bigg) ^2 \Bigg]\Bigg.\\
	& \quad\quad\quad\quad\quad\quad+  \Bigg. 2Y_i\bigg( \sum_{j=1}^J w_j \hat{Y}_{j,-i} - \sum_{j=1}^J w_j \hat{Y}_j        \bigg)       \Bigg\}^2 . 
\end{aligned}
\end{eqnarray}

Using $(a+b)^2\le 2a^2 + 2b^2$, $(a^2 - b^2) = (a-b)(a+b)$ and Cauchy-Schwarz, \eqref{2.4.2-15} is bounded from above by
\begin{eqnarray} \label{2.4.2.15'}
\begin{aligned}
	& \quad\quad\quad 2 E_{\bm{Y}} \left[\left( \sum_{j=1}^J w_j \hat{Y}_{j} -  \sum_{j=1}^J w_j \hat{Y}_{j,-i}       \right) ^2 \left( \sum_{j=1}^J w_j \hat{Y}_{j} +  \sum_{j=1}^J w_j \hat{Y}_{j,-i} \right) ^2\right]\\
	&\quad\quad\quad + 8\sqrt{E_{\bm{Y}} Y_i^4} \left[E_{\bm{Y}} \left( \sum_{j=1}^J w_j \left( \hat{Y}_{j} -  \hat{Y}_{j,-i}  \right) \right) ^4\right]^{1/2}. 
\end{aligned}	
\end{eqnarray}

Re-arranging in \eqref{2.4.2.15'} and using Cauchy-Schwarz and Jensen's inequality repeatedly gives the upper bound
\begin{eqnarray}
	&& 2 E_{\bm{Y}} \left[ \sum_{j=1}^J w_j \left( \hat{Y}_{j} -  \hat{Y}_{j,-i}\right)  \right] ^2 \left[ \sum_{j=1}^J w_j \left( \hat{Y}_{j} +  \hat{Y}_{j,-i}\right)  \right] ^2\nonumber \\
	&&\quad+ 8 \sqrt{E_{\bm{Y}} Y_i^4} \left[E_{\bm{Y}} J^2 \left( \sum_{j=1}^J w_j^2\left( \hat{Y}_{j} -  \hat{Y}_{j,-i}         \right) ^2         \right) ^2\right]^{1/2}\nonumber \\
	&&\le 2 J \left( \sum_{j=1}^J w_j^2\right) ^2  
	   \left[   E_{\bm{Y}}  \sum_{j=1}^J  \left( \hat{Y}_{j} +  \hat{Y}_{j,-i}  \right) ^4 \ \right]^{1/2} 
	   \left[   E_{\bm{Y}}  \sum_{j=1}^J  \left( \hat{Y}_{j} -  \hat{Y}_{j,-i}         \right) ^4        \right]^{1/2}\nonumber\\
	 && \quad+ 8 J \sqrt{E_{\bm{Y}} Y_i^4} \left( \sum_{j=1}^J w_j^4 \right) ^{1/2}  \left[   E_{\bm{Y}} \sum_{j=1}^J  \left( \hat{Y}_{j} +  \hat{Y}_{j,-i}\right)^4  \right]^{1/2} .\nonumber
\end{eqnarray}
This last expression equals
\begin{eqnarray}
	&&   \left[2 J \left( \sum_{j=1}^J w_j^2\right) ^2  
	   \left(  E_{\bm{Y}}  \sum_{j=1}^J  \left( \hat{Y}_{j} +  \hat{Y}_{j,-i}         \right) ^4  \right)^{1/2} \right.\nonumber\\
	   &&\quad\quad\quad\quad\quad\left. +8 J \sqrt{E_{\bm{Y}} Y_i^4} \left( \sum_{j=1}^J w_j^4 \right) ^{1/2}  \right] 
	   \times
	 \left[E_{\bm{Y}}   \sum_{j=1}^J \left( \hat{Y}_{j} -  \hat{Y}_{j,-i}         \right) ^4 \right]^{1/2},  \label{2.4.2.16}
\end{eqnarray}
where $\hat{Y}_{j} = E_j (Y_{n+1} \mid \bm{Y})$ and $\hat{Y}_{j, -i} = E_j (Y_{n+1} \mid \bm{Y}_{-i})$ by Assumption (iii).  
Next, we show the
second factor in \eqref{2.4.2.16} goes to zero as $n \rightarrow \infty$.


{\it Step 3, Part 1:  The fourth power inside the expectation in \eqref{2.4.2.16} goes to zero almost everywhere.}
First observe that  there is a $\sigma$-field $\sigma_\infty$ so that $\sigma(Y_1, \ldots , Y_n) \nearrow \sigma_\infty$
as $n \rightarrow \infty$.  
Similarly, $\sigma(Y_1, \ldots , Y_{n-1})   \nearrow \sigma_\infty$.  Indeed, the $\sigma$-field generated by any $n-1$ of
the $Y_{-i}$'s converges to $\sigma_\infty$ for each $j$.   Since $E_jY_{n+1}$ is finite for all $j$,
the martingale convergence theorem, \citet{billingsley}  (Theorem 35.6 on p. 499), gives that, for each $j$, 
\begin{eqnarray}
E_j(Y_{n+1} \mid \bm{Y}) , E_j(Y_{n+1} \mid \bm{Y}_{-i}) \rightarrow E_j(Y_{n+1} \mid \sigma_\infty)
\nonumber
\end{eqnarray}
almost everywhere in $M_j$ as $n \rightarrow \infty$.  Hence,
 \begin{eqnarray}
E_j(Y_{n+1}\mid\bm{Y}) -  E_j(Y_{n+1}\mid\bm{Y}_{-i}) \rightarrow 0,
\label{AEconverge}
\end{eqnarray}
almost everywhere in $M_j$ for each of the $x_i$'s.
However, since all the models have the same sets of measure zero in the underlying measure space, convergence
\eqref{AEconverge} holds almost everywhere in the limit over the mixture of all $J$ models, i.e., with respect to $p(\bm{y})$
in a limiting sense.  Thus, the fourth powers also converge to zero almost everywhere in the mixture
of the $J$ models.

{\it Step 3, Part 2:  The fourth powers inside the expectation in \eqref{2.4.2.16} are uniformly integrable.}  By the Cauchy-Schwarz inequality,
\begin{eqnarray}	\label{upperbd}
	&&\sum_{j=1}^J \left( E_j(Y_{n+1}\mid\bm{Y}) -  E_j(Y_{n+1}\mid\bm{Y}_{-i})  \right) ^4 \nonumber\\
	&\le& 4 \sum_{j=1}^J \left( E_j(Y_{n+1}^{2}\mid\bm{Y}) +  E_j(Y_{n+1}^{2}\mid\bm{Y}_{-i})  \right) ^2\nonumber\\
	&\le& 8  \sum_{j=1}^J \left( E_j(Y_{n+1}^{4}\mid\bm{Y}) +  E_j(Y_{n+1}^{4}\mid\bm{Y}_{-i})\right).
\end{eqnarray}
So, by Assumption (i) and \citet{billingsley} (Lemma p. 498), we have that the right hand side is a uniformly integrable
sequence since the sequence of $x_i$'s in $K_1$ is regarded as a countable collection of fixed design points.  Assumption (ii) together with Step 3, Part 1,
gives
\begin{eqnarray}
	\quad E_{\bm{Y}} \left[ \sum_{j=1}^J \left( E_j(Y_{n+1}\mid\bm{Y}) -  E_j(Y_{n+1}\mid\bm{Y}_{-i})  \right) ^4 \right] \rightarrow 0 \mbox{ as } n\rightarrow\infty  \label{2.4.2.22}
\end{eqnarray}
uniformly in the $x_i$'s.  Thus, the bound on the $i$-th term in \eqref{2.4.2.15} is independent of $i$ and so is a valid upper bound that
goes to zero for all $n$ terms in \eqref{2.4.2.15}, therefore bounding the average.

To conclude the proof, we show that the first factor in \eqref{2.4.2.16} is uniformly bounded as a function of $n$ for $x_1, \dots , x_n$ in a compact set.  
Indeed, by Assumption (i), the expectation $E_{\bm{Y}} Y_i^4$ is
\begin{eqnarray}
	\quad\quad\quad\int y^4 \left( \sum_{j=1}^J \pi_j \int p_j(y\mid x_i, \theta_j) w_j(\theta_j) d\theta_j \right) dy
	= \sum_{j=1}^J \pi_j E_j(Y^4) < \infty. \label{2.4.2.23'} 
\end{eqnarray} 
Moreover, $\sum_{j=1}^J  \left( \hat{Y}_{j} +  \hat{Y}_{j,-i} \right) ^4$ is uniformly integrable for the same reason as
\eqref{upperbd} is and converges to
$$
16 \sum_{j=1}^J [E_j (Y_{n+1}\mid \sigma_\infty)]^4.
$$
So, by \citet{billingsley} (Theorem 25.12 p. 361)  $16 \sum_{j=1}^J [E_j (Y_{n+1}\mid \sigma_\infty)]^4$ is integrable and
$E_{\bm{Y}} \sum_{j=1}^J  \left( \hat{Y}_{j} +  \hat{Y}_{j,-i} \right) ^4 \rightarrow E_{\bm{Y}} 16 \sum_{j=1}^J [E_j (Y_{n+1}\mid \sigma_\infty)]^4$
giving that the first factor in \eqref{2.4.2.16} is bounded.

Taken together, these statements imply that the second term in \eqref{2.4.2.14} goes to $0$ as $n\rightarrow\infty$. 

\end{proof}
The result in Theorem \ref{thm1} remains true for squared error loss if we use plug-in predictors of the form
$$
\hat{Y}_{j} = E_{\hat{\theta}_j(\bm{Y})}(Y_{n+1}) = \int y_{n+1} p_{\hat{\theta}_j(\bm{Y})}(y_{n+1}) d y_{n+1},
$$
where $\hat{\theta}_j$ is any consistent estimator for $\theta_j$ in $M_j$ 
rather than Bayes predictors as in Assumption (iii). 
This assertion is in the following. 
\begin{thm}\label{thm2}
Let $\ell(z, a) = (z-a)^2$ denote squared error loss.  Assume
\begin{description} 
\item (i) For any $j=1, \ldots , J$ and any pre-assigned $\epsilon>0$,  
	$$
	E_j(Y^{4+\epsilon}) = \int \int y^{4+\epsilon} p_j(y\mid x, {\theta}_j) w_j({\theta}_j) d\theta_j d y < \infty, 
	$$ 
	and $E_j (Y^{4+\epsilon})$ is continuous for $x \in K_1$,
\item (ii) For each $j = 1, \ldots , J$, the conditional densities $p_j( y \mid x, \theta_j)$ are equicontinuous for $ x \in K_1$ for
each $y$ and $\theta_j \in K_2$, and,
\item (iii) For each $j\ = 1, \ldots , J$, let the plug-in predictor
	$$
	\hat{Y}_{j} = E_{\hat{\theta}_j(\bm{Y})}(Y_{n+1}) = \int y_{n+1} p_{\hat{\theta}_j(\bm{Y})}(y_{n+1}) d y_{n+1}, 
	$$ 
	be used to generate predictions at the $n+1$ step, where $\hat{\theta}_j(\bm{Y})$ is a consistent estimator for $\theta_j$, and
\item (iv) For each $j = 1, \ldots , J$, $E_{\theta_j}(Y_{n+1})$ and $E_{\theta_j}(Y_{n+1}^{4})$ are continuous for $\theta_j\in\Theta_j$, 
where $\Theta_j \subset K_2$ is a 
compact parameter space for $M_j$.
\end{description}
Then, for any action $a(\bm{Y})= \sum_{j=1}^J w_j \hat{Y}_j$, $w_j \in \mathbb{R}$ for all $j$, we have
\begin{equation}
\nonumber
	\int \ell(y_{n+1},a(\bm{Y})) p(y_{n+1}\mid\bm{Y}) d{y}_{n+1} - \frac{1}{n} \sum_{i=1}^n \ell(Y_i,a(\bm{Y}_{-i})) \stackrel{L^2}{\rightarrow} 0 \mbox{ as } n\rightarrow\infty.
\end{equation}
\end{thm}

\begin{proof}
Fix a countable sequence $x_1, x_2, \ldots \in K_1$.
Recall expression \eqref{2.4.2.14}.  The first term is treated much the same as in Theorem \ref{thm1}.  The main difference is that 
Assumption (iv) can be used to establish uniform integrability of $E_{\hat{\theta}_j(\bm{Y})}Y_{n+1}^4$ because
it gives $\sup_n E_{\bm{Y}} \left[  E_{\hat{\theta}_j(\bm{Y}) }(Y_{n+1}^4) \right]^2 < \infty$.

Showing the second term in \eqref{2.4.2.14} goes to zero uses the plug-in estimator version of \eqref{2.4.2.16}.  Showing it goes to
zero requires 
\begin{eqnarray} 
\label{extrastep}
	E_{\hat{\theta}_j(\bm{Y})} (Y_{n+1}) \stackrel{L_1}{\rightarrow} E_{\theta_j}(Y_{n+1}) \mbox{ as } n\rightarrow\infty
\end{eqnarray}
which uses Assumption (iv) for the first moment (rather than a martingale argument) and a uniform integrability argument similar
to that used in Step 3 of the proof of Theorem \ref{thm1}.

\end{proof}

Unsurprisingly, the conclusions of Theorems \ref{thm1} and \ref{thm2}  continue to hold if the squared error is replaced by
the absolute error $\ell(z,a) =  |z-a|$.   We state the Bayes and plug-in versions for absolute error in the following.

\begin{thm}\label{thm3}
Let $\ell(z, a) = |z-a|$ denote absolute error loss.  Assume
\begin{description} 
\item (i) For any $j=1, \ldots , J$ and any pre-assigned $\epsilon>0$,  
	$$
	E_j(Y^{2+\epsilon}) = \int \int y^{2+\epsilon} p_j(y\mid x, {\theta}_j) w_j({\theta}_j) d\theta_j d y < \infty, 
	$$ 
	and $E_j (Y^{2+\epsilon})$ is continuous for $x \in K_1$,
\item (ii) For each $j = 1, \ldots , J$, the conditional densities $p_j( y \mid x, \theta_j)$ are equicontinuous for $x \in K_1$ for
each $y$ and $\theta_j \in K_2$, and,
\item (iii) For each $j = 1, \ldots, J$, let either the Bayes or the plug-in predictor from Theorem \ref{thm1} or Theorem \ref{thm2}, respectively,
be used to generate predictions at the $n+1$ time step. 
\item (iv) If plug-in predictors are chosen in Assumption (iii), then assume in addition that
for each $j = 1, \ldots , J$, $E_{\theta_j}(Y_{n+1})$ and $E_{\theta_j}(Y_{n+1}^{2})$ are continuous as functions of $\theta_j\in\Theta_j$, where $\Theta_j \subset K_2$ is a 
compact parameter space for $M_j$.
\end{description}
Then, for any action $a(\bm{Y})= \sum_{j=1}^J w_j \hat{Y}_j$, $w_j \in \mathbb{R}$ for all $j$, we have
\begin{equation}
\nonumber
	\int \ell(y_{n+1},a(\bm{Y})) p(y_{n+1}\mid\bm{Y}) d{y}_{n+1} - \frac{1}{n} \sum_{i=1}^n \ell(Y_i,a(\bm{Y}_{-i})) \stackrel{L^2}{\rightarrow} 0 \mbox{ as } n\rightarrow\infty.
\end{equation}
\end{thm}

\begin{proof}
Fix a countable sequence $x_1, x_2, \ldots \in K_1$.
As before, we recall \eqref{2.4.2.14}.  Verifying that the first term goes to zero is the same as in the proof of Theorem \ref{thm1} for the Bayes 
posterior mean predictor and can be modified as explained in the proof of Theorem \ref{thm2} for the plug-in predictor.

The $i$-th summand in the second term in \eqref{2.4.2.14} is bounded by
\begin{eqnarray} 
	&& E_{\bm{Y}} \left(  | Y_i-\sum_{j=1}^J w_j \hat{Y}_j | -  | Y_i- \sum_{j=1}^J w_j \hat{Y}_{j,-i}|\right)^2  .
 \label{2.4.2.33}
\end{eqnarray}  
Since $(|a|-|b|)^2\le (a-b)^2$, (\ref{2.4.2.33}) is bounded from above by
\begin{eqnarray} 
	 E_{\bm{Y}} \left[  \sum_{j=1}^J w_j \left( \hat{Y}_{j} -  \hat{Y}_{j,-i} \right) \right] ^2. \label{2.4.2.34}
\end{eqnarray} 
Using Cauchy-Schwarz inequality, (\ref{2.4.2.34}) is upper bounded by	
\begin{eqnarray*}
\begin{aligned}
	&E_{\bm{Y}} \left\lbrace  \left( \sum_{j=1}^J w_j^2\right) \left[ \sum_{j=1}^J \left(  \hat{Y}_{j} -  \hat{Y}_{j,-i} \right)^2 \right]    \right\rbrace   
	=\left( \sum_{j=1}^J w_j^2\right)  \sum_{j=1}^J E_{\bm{Y}} \left(  \hat{Y}_{j} -  \hat{Y}_{j,-i} \right)^2. \label{2.4.2.35}
\end{aligned}
\end{eqnarray*}  

If the Bayes predictors are used, the last expression can be controlled by Step 3 in the proof of Theorem \ref{thm1}.
If the plug-in predictors are used, then the last expression can be controlled by the extra step mentioned in
the proof of Theorem \ref{thm2}, see \eqref{extrastep}.

\end{proof}

The so-called log-loss is qualitatively different from squared error or absolute error because log-loss can be 
positive or negative.  It is therefore better regarded  as a utility function even though in some cases it can be
physically interpreted as the code length function for a Shannon code.   Here we extend our results to the
log-utility to verify that cross-validation continues to remain an asymptotically Bayes procedure.
Now, an action $a$ is of the form
\begin{equation}
\label{BayesAct}
	a(Y_{n+1}\mid\bm{Y}) = \sum_{j=1}^J w_j p_j (Y_{n+1}\mid\bm{Y}),
\end{equation}
and the corresponding log-utility is
\begin{equation}
	u(Y_{n+1}, a(Y_{n+1}\mid\bm{Y})) = \log \left[ \sum_{j=1}^J w_j p_j (Y_{n+1}\mid\bm{Y})\right].
\end{equation}

As shown in our next result, cross-validation approximates the posterior expected utility of  
the Bayes action of the form \eqref{BayesAct} or the plug-in action of the form 
$a(Y_{n+1}\mid\bm{Y}) = \sum_{j=1}^J w_j p_{\hat{\theta}_j(\bm{Y})}(Y_{n+1})$ where
 $\hat{\theta}_j(\bm{Y})$ is a consistent estimator of $\theta_j \in K_2$ in $M_j$. 

\begin{thm}\label{thm4} Let $u(Y_{n+1},a(\bm{Y})) = \log \left[ \sum_{j=1}^J w_j p_j (Y_{n+1}\mid\bm{Y})\right]$ denote log-utility.
Assume
\begin{description} 
\item (i) For each $j = 1, \ldots, J$, there is a function $B_j(\cdot)$ so that
	$$
	\sup_{\bm{Y}} |\log p_j(Y_{n+1}\mid\bm{Y})| \le B_j(Y_{n+1}) < \infty,
	$$
	$B_j(\cdot)$ is independent of $x_1, x_2, \ldots$, and 
	$$
	E[g(Y_{n+1})] < \infty,
	$$
	where 
\begin{eqnarray*}
 \begin{aligned}
 	g(Y_{n+1}) = \max\left\{  \left(\log \sum_{j=1}^J w_j  e^{-B_j(Y_{n+1})}\right)^4, \left(\log  \sum_{j=1}^J w_j  e^{B_j(Y_{n+1})} \right)^4\right\}.
\end{aligned}
 \end{eqnarray*}
\item (ii) For each $j = 1, \ldots , J$, the conditional densities $p_j( y \mid x, \theta_j)$ are equicontinuous in $x$ for
each $y$ and $\theta_j \in \Theta_j \subset K_2$, and the predictive densities $p_j(y\mid\bm{Y}) $ within the $j$-th model
are uniformly equicontinuous in $y$.
\item (iii) For each $j = 1, \ldots, J$, let the Bayes action \eqref{BayesAct}
be used to generate predictions at the $n+1$ time step. 
\end{description}
Then,  we have
\begin{equation}
\nonumber
	\int u(y_{n+1},a(\bm{Y})) p(y_{n+1}\mid\bm{Y}) d{y}_{n+1} - \frac{1}{n} \sum_{i=1}^n u(Y_i,a(\bm{Y}_{-i})) \stackrel{L^2}{\rightarrow} 0 \mbox{ as } n\rightarrow\infty.
\end{equation}
If $p_{\hat{\theta}_j(\bm{Y})}(y_{n+1})$ where $\hat{\theta}_j(\bm{Y})$ is a consistent estimator of $\theta_j$ is used instead of $p_j(y_{n+1}\mid\bm{Y})$, 
the result still holds.
\end{thm}
\begin{proof}
The structure of the proof  is similar to that of Theorem \ref{thm1} or Theorem \ref{thm2}.  For both Bayes
and plug-in predictors the main difference between the proof of this theorem and the proofs of Theorems \ref{thm1} and \ref{thm2}
is that in {\it Step 1} Assumption (i) is used to establish uniform integrability  (and convergence to zero) of the first term in \eqref{2.4.2.14}.
Showing that the second term in  \eqref{2.4.2.14}  goes to zero for the Bayes predictors requires Assumptions (i) and (ii) to get
\begin{eqnarray*}
	p_j(y_{n+1}\mid\bm{Y})  \rightarrow p_j(y_{n+1}\mid \sigma_\infty)
\end{eqnarray*}
almost everywhere as $n\rightarrow\infty$ to set up an application of the dominated convergence theorem.  For the plug-in
predictors, we need the extra step described in the proof of Theorem \ref{thm2}, see \eqref{extrastep}.
\end{proof}
\section{Derivation of stacking weights}
\label{derivation}

Suppose we have $J$ predictors $\hat{Y}_1,\cdots,\hat{Y}_J$ from distinct models.  Then we might seek weights 
$\hat{w}_j$, using the training data, so as to form a model average prediction at $x_{new}$ of the form
\begin{eqnarray}
	\hat{y}(x_{new}) = \sum_{j=1}^J \hat{w}_j \hat{y}_j (x_{new}).
	\label{modavg}
\end{eqnarray}  
From a Bayesian point of view, one should find the action that minimizes the posterior risk (or maximizes the posterior expected utility)
given the data $\bm{y}$.  Theorem \ref{thm1} shows that  the posterior risk is asymptotically equivalent to 
\begin{eqnarray*}
	 \frac{1}{n} \sum_{i=1}^n \ell(y_i,a(\bm{y}_{-i})) = \frac{1}{n} \sum_{i=1}^n  \left(  y_i - \sum_{j=1}^J w_j \hat{y}_{j,-i} (x_i)\right)^2,
\end{eqnarray*} 
when $\ell$ is squared error loss.  Ignoring the $(1/n)$ and minimizing over the $\hat{w}_j$'s gives the same expression as \eqref{stackcrit}.
That is, the stacking weights are asymptotically Bayes optimal -- the precise form of optimality given by the
constraints imposed on the $w_j$'s -- and can be used in \eqref{modavg} to give the
stacking predictor.  This formalizes the heuristic approximations used in \citet{clyde}.

There are two constraints on the $w_j$'s that are commonly used.  One is the `sum to one' constraint that requires $\sum_j w_j = 1$ (see \citet{clyde}) and
the other is the non-negativity constraint that requires all $w_j \geq 0$ (see \citet{breiman}). Removing the non-negativity constraint and relaxing the sum-to-one constraint to 
a sum-to-$m$ constraint give the following.

\begin{thm}\label{thm3.2n}
	The weights $w_1, \ldots , w_J$ achieving
	\begin{eqnarray*}
		\min_w \sum_{i=1}^n \left( y_i - \sum_{j=1}^J w_j \hat{y}_{j,-i} (x_i)\right)^2 \mbox{ subject to } \sum_{j=1}^J w_j= m
	\end{eqnarray*} 
are of the form
	\begin{eqnarray*}
		\hat{w} \propto U^{-1} 1_J,
	\end{eqnarray*} 	 
where 
	\begin{eqnarray}\label{3.5n}
	\begin{aligned}
		U &= (u_{lj})_{J\times J}, \\
		u_{lj} &= \sum_{i=1}^n \left( \frac{y_i}{m} - \hat{y}_{j,-i} \right)\hat{y}_{l,-i}  - \sum_{i=1}^n \left( y_i - \hat{y}_{j,-i} \right) y_i, \\
		 1_J &= (1,\cdots,1)'.
	\end{aligned}
	\end{eqnarray} 	 
\end{thm}
\begin{proof}
This is a standard Lagrange multipliers problem.  Write the Lagrangian as
\begin{eqnarray*}
	L= - \sum_{i=1}^n \left(  y_i - \sum_{j=1}^J w_j \hat{y}_{j,-i} \right)^2 -\lambda_0 (\sum_{j=1}^J w_j - m).
\end{eqnarray*} 	

Then $\hat{w}$ is the solution of the following system,
\begin{eqnarray}
	\quad\quad\frac{\partial L}{\partial w_l} &= & 2  \sum_{i=1}^n \left(  y_i - \sum_{j=1}^J w_j \hat{y}_{j,-i} \right) \hat{y}_{l,-i} - \lambda_0 = 0 \mbox{ for } l = 1,\cdots,J, \label{3.2n}\\
	\frac{\partial L}{\partial \lambda_0} &=& \sum_{j=1}^J w_j - m = 0. \label{3.3n}
\end{eqnarray} 

From (\ref{3.2n}) and (\ref{3.3n}), we have
\begin{eqnarray*}
\begin{aligned}
	 &\sum_{i=1}^n \left(  y_i - \sum_{j=1}^J w_j \hat{y}_{j,-i} \right) \hat{y}_{l,-i}= \frac{\lambda_0}{2} \\
	 &\Rightarrow \sum_{i=1}^n y_i \hat{y}_{j,-i} - \sum_{j=1}^J w_j   \sum_{i=1}^n \hat{y}_{j,-i} \hat{y}_{l,-i} = \frac{\lambda_0}{2}\\
	 &\Rightarrow \frac{1}{m}\sum_{i=1}^n y_i \hat{y}_{j,-i} \sum_{j=1}^J w_j - \sum_{j=1}^J w_j   \sum_{i=1}^n \hat{y}_{j,-i} \hat{y}_{l,-i} 
	 	- \sum_{j=1}^J \sum_{i=1}^n (y_i - \hat{y}_{j,-i} ) y_i w_j \\
		&\quad\quad = \frac{\lambda_0}{2} - \sum_{j=1}^J \sum_{i=1}^n (y_i - \hat{y}_{j,-i} ) y_i w_j  \quad\mbox{ for } l = 1,\cdots,J
\end{aligned}
\end{eqnarray*} 
Since the right hand side does not depend on $l$, we have
\begin{eqnarray*}
\begin{aligned}
	 &\frac{1}{m}\sum_{i=1}^n y_i \hat{y}_{j,-i} \sum_{j=1}^J w_j - \sum_{j=1}^J w_j   \sum_{i=1}^n \hat{y}_{j,-i} \hat{y}_{l,-i} 
	 	- \sum_{j=1}^J \sum_{i=1}^n (y_i - \hat{y}_{j,-i} ) y_i w_j \propto 1.
\end{aligned}
\end{eqnarray*} 
Rearranging gives
\begin{eqnarray*}
\begin{aligned}
	 &w_1\left(  \frac{1}{m}\sum_{i=1}^n y_i \hat{y}_{l,-i} -  \sum_{i=1}^n \hat{y}_{1,-i} \hat{y}_{l,-i} - \sum_{i=1}^n (y_i - \hat{y}_{1,-i} ) y_i\right)\\
	 &+ w_2\left(  \frac{1}{m}\sum_{i=1}^n y_i \hat{y}_{l,-i} -  \sum_{i=1}^n \hat{y}_{2,-i} \hat{y}_{l,-i} - \sum_{i=1}^n (y_i - \hat{y}_{2,-i} ) y_i\right)\\
	 &\quad\vdots\\
	 &+ w_J\left(  \frac{1}{m}\sum_{i=1}^n y_i \hat{y}_{l,-i} -  \sum_{i=1}^n \hat{y}_{J,-i} \hat{y}_{l,-i} - \sum_{i=1}^n (y_i - \hat{y}_{J,-i} ) y_i\right) \propto 1,
\end{aligned}
\end{eqnarray*} 
 for $l = 1,\cdots,J$.

In matrix form, this system of equations is
\begin{eqnarray*}
\begin{aligned}
	U w  \propto 1_J,
\end{aligned}
\end{eqnarray*} 
where $U$ and $1_J$ are defined as in (\ref{3.5n}).
Therefore,  the solution is 
	\begin{eqnarray*}
		\hat{w} \propto U^{-1} 1_J,
	\end{eqnarray*} 	
which can be rescaled to satisfy the sum to $m$ constraint.

\end{proof}

\begin{cor} \label{cor1}
	If $m=1$, then 	the weights $w_1, \ldots , w_J$ achieving
	\begin{eqnarray*}
		\min_w \sum_{i=1}^n \left(  y_i - \sum_{j=1}^J w_j \hat{y}_{j,-i} (x_i)\right)^2 \mbox{ subject to } \sum_{j=1}^J w_j= 1
	\end{eqnarray*} 
are of the form
	\begin{eqnarray}\label{3.1'}
		\hat{w} \propto \left( \hat{e}'  \hat{e}\right)^{-1} 1_J,
	\end{eqnarray} 	 
where 
	\begin{eqnarray*}
		\hat{e} =  \left( y_i - \hat{y}_{j,-i}\right)_{n\times J}   \mbox{ and } 1_J = (1,\cdots,1)'.
	\end{eqnarray*} 	 
\end{cor}
\begin{rem}
This corollary is the result \citet{clyde} used.
\end{rem}

For contrast, let us solve \eqref{stackcrit} but without any sum  constraint (and without the
non-negativity constraint).  Now, the Lagrangian is 
\begin{eqnarray*}
	L= - \sum_{i=1}^n \left(  y_i - \sum_{j=1}^J w_j \hat{y}_{j,-i} \right)^2
\end{eqnarray*} 	
and $\hat{w}$ is the solution of the system of equations
\begin{eqnarray*}
	\frac{\partial L}{\partial w_l} &= & 2  \sum_{i=1}^n \left( y_i - \sum_{j=1}^J w_j \hat{y}_{j,-i} \right)  \hat{y}_{l,-i} = 0 \mbox{ for } l = 1,\cdots,J.
\end{eqnarray*} 
Therefore,
\begin{eqnarray}\label{3.5'}
\begin{aligned}
	 & \sum_{i=1}^n \hat{y}_{l,-i} \sum_{j=1}^J w_j \hat{y}_{j,-i} = \sum_{i=1}^n y_i \hat{y}_{l,-i} \\
	 & \Leftrightarrow  \sum_{j=1}^J \left(\sum_{i=1}^n \hat{y}_{l,-i}  \hat{y}_{j,-i}\right) w_j = \sum_{i=1}^n y_i \hat{y}_{l,-i}, \mbox{ for } l = 1,\cdots,J,
\end{aligned}
\end{eqnarray} 
or, in matrix form,
\begin{eqnarray*}
	T w = c,
\end{eqnarray*} 
where 
\begin{eqnarray}\label{3.5}
\begin{aligned}
	&T =   \left(\sum_{i=1}^n \hat{y}_{l,-i}  \hat{y}_{j,-i}\right)_{J \times J},\\
	& c = \left(\sum_{i=1}^n y_i \hat{y}_{1,-i},\cdots,\sum_{i=1}^n y_i \hat{y}_{J,-i}\right)'.
\end{aligned}
\end{eqnarray} 	 
Hence the solution to \eqref{stackcrit} without the sum to one constraint and without the non-negativity constraint is 
\begin{eqnarray*}
	\hat{w} = T ^{-1} c.
\end{eqnarray*} 

We summarize this in the following theorem.
\begin{thm}\label{thm3.2}
	The weights $w_1, \ldots , w_J$ achieving
	\begin{eqnarray*}
		\min_w \sum_{i=1}^n \left(  y_i - \sum_{j=1}^J w_j \hat{y}_{j,-i} (x_i)\right)^2
	\end{eqnarray*} 
are of the form
	\begin{eqnarray}\label{3.7}
		\hat{w} = T ^{-1} c,
	\end{eqnarray} 	 
where $T$ and $c$ are given in (\ref{3.5}).
In addition, if the $J$ predictors are orthonormal,
\begin{eqnarray*}
		\sum_{i=1}^n \hat{y}_{l,-i}  \hat{y}_{j,-i} = \delta_{l\ne j}  \quad(1 \mbox{ if } l\ne j \mbox{ and } 0 \mbox{ otherwise}),
\end{eqnarray*} 	
then $T = I$ and the solution becomes
\begin{eqnarray}\label{3.8}
		\hat{w}_j = \sum_{i=1}^n y_i \hat{y}_{j,-i}, \mbox{ for } j = 1,\cdots,J.
\end{eqnarray} 
\end{thm}

Note that the minimum in  Corollary \ref{cor1} with the sum to one constraint  is taken over a smaller set than that of Theorem \ref{thm3.2} without any
sum  constraint. 
So, when the stacking weights from the two cases both exist, we expect the latter to give better predictive performance
because the minimum in Theorem \ref{thm3.2} can only be smaller than the minimum in Corollary \ref{cor1}.  
Hence we do not favor imposing the sum to one constraint.  Indeed, we find in our computed examples that when a sum to one constraint 
gives better prediction, it is merely a happenstance from the more general optimization.  This is straightforward because
if we find the optimal weights from Theorem \ref{thm3.2} then we can use them to find $m = \sum_{j=1}^J w_j$ for use in Theorem \ref{thm3.2n}.

Using arguments similar to those used in the proof of Theorem \ref{thm3.2}, the following result extends Theorem \ref{thm3.2}
to a Hilbert space $\cal{H}$ equipped with an empirical inner product
\begin{eqnarray*}\label{3.9}
	\langle g, h \rangle_n = \frac{1}{n} \sum_{i=1}^n g(x_i)h(x_i) \quad\forall g, h \in \cal{H}.
\end{eqnarray*} 
\begin{thm}\label{thm3.3}
The weights $w_1, \ldots , w_J$ achieving
	\begin{eqnarray*}\label{3.10}
		\min_w \sum_{i=1}^n  \left(  y(x_i) - \sum_{j=1}^J w_j \hat{f}_{j, -i}(x_i) \right)^2,
	\end{eqnarray*} 
where $y$ and $\hat{f}_{j, -i}$, $j=1,\cdots,J$, belong to $\cal{H}$, are of the form
	\begin{eqnarray*}\label{3.11}
		\hat{w} = T ^{-1} c,
	\end{eqnarray*} 	 
where $T$ and $c$ are of the same form as \eqref{3.5}.

\end{thm}

As $n\rightarrow\infty$, there are conditions that ensure the empirical inner product $\langle g, h \rangle_n$
converges uniformly to the inner product $\langle g, h \rangle = \int g(x)h(x) dx$ of the $\cal{H}$ space, see \citet{geer}.
Therefore, as $n$ increases we can approximate the empirical inner product by the $\cal{H}$ inner product
and the results in Theorem \ref{thm3.3} will remain true. 
\section{What models should we put in the stack?}
\label{basisselection}

Here, we show that the intuition of \citet{breiman} that the models to be stacked should be as different as
possible is only partially correct.  What matters about the models to be stacked is that they be independent.
The extra `difference' amongst models from imposing orthogonality is not actually helpful in terms of reducing
the error criterion \eqref{stackcrit}.  We show this for models constructed in general Hilbert spaces
of functions and then provide one possible answer for how to construct the models in a Hilbert space to be stacked.
Note that this restriction limits us to ${\cal{M}}$-complete problems. In general, in ${\cal{M}}$-open problems
one cannot assume the regression function is in a Hilbert space. However, in our ${\cal{M}}$-open example in Subsec. \ref{seceg}
we did not find orthonormality of a basis gave better predictions.

\subsection{The error depends only on the span of the model list.}
\label{modlistspan}
In the last section, we saw that releasing the
sum to one constraint can only reduce the error criterion and from our example in the
introduction we saw that this constraint can often be genuinely harmful.  This argument is particularly strong outside
${\cal{M}}$-closed settings where model mis-specification is  always present.  

Our first result shows that given a set of models to stack, the error  depends only on the span of the models;
requiring that the models to be stacked be orthogonal as well as independent does not reduce the error.
Our  result is the following.

\begin{thm}\label{thm4.2}
Let  ${\cal{H}}$ be a Hilbert space with inner product denoted
$\langle \cdot  , \cdot \rangle$.  Let  ${\cal{M} }= \{ f_1,\cdots, f_J \}$ and ${\cal{M}}^\prime = \{ f^\prime_1,\cdots, f^\prime_{J^\prime} \}$ be sets of elements
from ${\cal{H}}$ with minima $Q_{\min}^{\cal{M}}$ and $Q_{\min}^{{\cal{M}}^\prime}$ for \eqref{stackcrit},
respectively. 
Denote the span of a set of elements in ${\cal{H}}$ by $\langle \cdot \rangle$.   Then,
if $\langle {\cal{M}} \rangle = \langle  {\cal{M}}^\prime \rangle$,
\begin{eqnarray*}\label{4.7}
		Q_{\min}^{\cal{M}} = Q_{\min}^{\cal{M}'},
\end{eqnarray*} 
i.e., the stacking error only depends on the span of the predictors.
\end{thm}

\begin{proof}  This involves routine manipulations with Hilbert spaces, see the
Appendix for details.
\end{proof}

Theorem \ref{thm4.2} means that given a fixed  subspace $S \subset {\cal{H}}$, any basis for $S$ is as good as any other
for forming a stacking predictor. 
So, we are free to choose whichever basis is most convenient.


Note that Theorem \ref{thm4.2} only applies  in the absence of constraints on the coefficients $w_j$.  Indeed,
the conclusion may be false if constraints are imposed.  Let $J=J'$, ${\cal{M}} = \{ \hat{y}_j = (\hat{y}_{j,-1},\cdots, \hat{y}_{j,-n})', j = 1,\cdots,J\}$ be an orthogonal basis, and  ${\cal{M}'}  = \{ \hat{y}_j' = (\hat{y}_{j,-1}' ,\cdots, \hat{y}_{j,-n}')', j = 1,\cdots,J\}$ be any basis 
of $\langle {\cal{M}}^\prime \rangle = \langle  {\cal{M}} \rangle$. Then, under the sum to one constraint on $\cal{M}$ we have
\begin{eqnarray}\label{4.9nn}
\begin{aligned}
	Q_{\min}^{\cal{M}} &= \left\| y - \sum_{j=1}^J   \hat{w}_j  \hat{y}_j \right\|^2\\
	&= \left\| \sum_{j=1}^J \langle y, \hat{y}_j \rangle \hat{y}_j +  \sum_{j=J+1}^n \langle y, e_j \rangle e_j - \sum_{j=1}^J   \hat{w}_j  \hat{y}_j \right\|^2\\
	&= \left\| \sum_{j=1}^J \left( \langle y, \hat{y}_j \rangle -  \hat{w}_j \right) \hat{y}_j    \right\|^2+ \| y_2 \|^2,
\end{aligned}
\end{eqnarray} 	 
where $\hat{w}$ is now the solution in Corollary \ref{cor1},  $\{ e_j, j=J+1,\cdots, n\}$ are complement vectors of 
$\{ \hat{y}_j, j = 1,\cdots,J\}$ to form an orthonormal basis of $\mathbb{R}^n$, and $y = y_1+y_2 = \sum_{j=1}^J \langle y, \hat{y}_j \rangle \hat{y}_j +  \sum_{j=J+1}^n \langle y, e_j \rangle e_j $. Similarly, for $\cal{M}'$ we have
\begin{eqnarray}\label{4.10nn}
\begin{aligned}
	Q_{\min}^{\cal{M}'} 	= \left\| \sum_{j=1}^J \left( \alpha_j -  \hat{w}_j' \right) \hat{y}_j'    \right\|^2+ \| y_2 \|^2,
\end{aligned}
\end{eqnarray} 	 
where $\hat{w}'$ is the solution in Corollary \ref{cor1} and $y = y_1+y_2 = \sum_{j=1}^J \alpha_j  \hat{y}_j' +  \sum_{j=J+1}^n \langle y, e_j \rangle e_j $. Obviously, from (\ref{4.9nn}) and (\ref{4.10nn}), 
it is possible for $Q_{\min}^{\cal{M}} < Q_{\min}^{\cal{M}'}$ or $Q_{\min}^{\cal{M}}  > Q_{\min}^{\cal{M}'}$.
This can be seen from the following example.  
Let $J=J^\prime=1$ and $\hat{y}'  = k \hat{y} $, then $\hat{w}=\hat{w}'=1$ and  $\alpha = \langle y, \hat{y} \rangle / k$. Hence    $Q_{\min}^{\cal{M}} = ( \langle y, \hat{y} \rangle-1)^2 + \| y_2 \|^2$ and $Q_{\min}^{\cal{M}'} = ( \langle y, \hat{y} \rangle-k)^2 + \| y_2 \|^2$.
So, by careful choice of $k$, $Q_{\min}^{\cal{M}}$ can be larger than $Q_{\min}^{\cal{M}'}$ or the reverse.


To reinforce Theorem \ref{thm4.2}, we observe that reducing the dimension of the span of the predictors
can only increase the error criterion.

\begin{thm}
\label{thm4.4n} 
Let ${\cal{M} } = \{ f_1,\cdots, f_J \} $ be a basis
and ${\cal{N}} = \{ f_1,\cdots, f_{J-1} \}$.  Let $Q_{\min}^{\cal{M}}$ and $Q_{\min}^{\cal{N}}$ be the 
minima of \eqref{stackcrit} corresponding to $\cal{M}$ and $\cal{N}$, respectively. Then,
\begin{eqnarray}\label{4.9n}
		Q_{\min}^{\cal{M}} \le Q_{\min}^{\cal{N}}.
\end{eqnarray} 
\end{thm}

\begin{proof}
This involves relatively routine manipulations with Hilbert spaces, see the
Appendix for details.
\end{proof}

Taken together,  Theorem \ref{thm4.2} and Theorem \ref{thm4.4n} tell us that the predictors being stacked should be different from 
one another in the sense of being independent (but not necessarily orthogonal) and that the stacking error \eqref{stackcrit}
is a non-increasing function of the span of the predictors.  Thus, when choosing predictors to stack,
there is a tradeoff between the number of predictors and their proximity to a true model assuming one exists.
That is, using more predictors will generally be helpful, but using fewer, better predictors can easily outperform
many, weaker predictors.

\subsection{Optimal choice of predictors to stack.}
\label{optchoice}
Having seen that both the number of basis elements and the proximity of a linear combination of them to a true function (if a true function exists) 
can be important we want to choose the basis elements effectively.
The results in Subsec. \ref{modlistspan} mean that, without loss of generality, we can limit our search to orthogonal bases.
Hence, in this subsection, we  propose a data-driven method to choose an optimal number of basis elements even if the set of basis elements is not unique.

Assume we have an orthonormal basis for a space  $\langle \{  e_1,\cdots, e_J   \}\rangle$ then for each $J'\le J$ we can form
\begin{eqnarray}\label{4.8}
\quad\quad\quad		\hat{y}_{J',\sigma_k,\lambda,i}(x_{\sigma_k(i)}) = \sum_{j=1}^{J'} \langle  e_j, \hat{f}_\lambda (\cdot\mid x_{\sigma_k(1)},\cdots,x_{\sigma_k(i-1)}) \rangle
		e_j(x_{\sigma_k(i)}),
\end{eqnarray} 
where $\hat{f}_\lambda$  is an estimate of the true predictor, for instance from the Nadaraya-Watson nonparametric regression estimator, see \citet{nadaraya} and 
\citet{watson}, $\lambda$ is a tuning parameter,  and 
$\sigma_k$  for $k=1, \ldots , K$ is a collection of independent permutations of $\{1,\cdots,n\}$.
Then, in principle, we can find 
\begin{eqnarray}\label{4.9}
\begin{aligned}
	&\{J_{opt}, basis_{opt}\} \\
	&= \arg\min_{J', basis} \sum_{k=1}^K \sum_{i=1}^n \left(   \hat{y}_{J',\sigma_k, \hat{\lambda},i}(x_{\sigma_k(i)}) - y(x_{\sigma_k(i)})    \right)^2 ,
\end{aligned}	
\end{eqnarray} 
where $\hat{\lambda}$ is an estimator for the tuning parameter and \textit{basis} is a variable varying over the possible orthonormal bases for
subspaces of $\langle \{  e_1,\cdots, e_J   \}\rangle$.

The idea is that \eqref{4.9} is a sort of variance-bias expression that can be minimized to find the right number of basis elements. 
Minimizing in \eqref{4.9}  means we are preventing the number of basis elements from being too small (high bias)
or too large (high variance).  This is embedded in
\eqref{4.9} because it uses $K$ independent orderings of the data and sequential predictive error (since we 
want $ \hat{y}_{J',\sigma_k, \hat{\lambda},i}(x_{\sigma_k(i)})$ at stage $\sigma_k(i)$ to use only the data $x_{\sigma_k(1)},\cdots,x_{\sigma_k(i-1)}$). 
Averaging over the permutations of the data points as they appear in the sequential predictive error means that it is reasonable
to regard the empirical optimum as close to an actual optimum if it exists.  In fact, in $\cal{M}$-open settings, it does not make
sense to take limits of \eqref{4.9}, so it is hard to prove theory.  
What is feasible is to seek a
$\{J_{opt}, basis_{opt}\}$ that makes \eqref{4.9} small.  This can be done numerically by a stochastic search provided we have a way
to propose basis vectors.

A simplification of \eqref{4.9} is to replace the sequential prediction with some form of cross-validation; this saves
on computing time.  

Our method for data-driven random generation of basis elements is simple.  Draw $J$ bootstrap samples from the data of size $n$.
For each bootstrap sample, define a basis element from the Nadaraya-Watson estimator.  This gives $J$ candidates
for basis vectors.  Next, apply Gram-Schmidt orthonormalization to form an orthonormal basis.
In some cases, two or more of these basis elements may be so close as to be de facto the same.  When this occurs,
we reject the results and repeat the procedure from the beginning until we get $J$ orthonormal basis elements. 
Note that any technique for nonparametric regression can be used in place of Nadaraya-Watson.  In Sec.
\ref{computedexamples} we also use Gaussian process priors, see \citet{rasmussen}, 
to generate function estimates that can be used as basis elements.

In addition, we may save on computing time by avoiding
having to permute over basis elements if they have a natural ordering e.g., Fourier bases are ordered by frequency, 
Legendre polynomials are ordered by degree etc.  
Although there are many ways to order randomly generated basis elements, here we order them by using the size
of regions in real spaces of dimension $\dim(x)$ as follows. 
Given the randomly generated orthonormal basis $\{ \hat{e}_1, \ldots , \hat{e}_J \}$ we first form the nonparametric
regression function $\hat{f}_{\hat{\lambda}} = \hat{f}_{\hat{\lambda}}(\cdot \mid (x_1, y_1), \ldots , (x_n, y_n) )$ using all the data.  Then,
we define the ordered basis $[e_1^*,\cdots,e_J^*]$  by the criterion
\begin{eqnarray*}
	| SA(\hat{f}_{\hat{\lambda}}) - SA (e_1^*)| \le \cdots \le | SA (\hat{f}_{\hat{\lambda}}) - SA (e_J^*)|,
\end{eqnarray*} 
where $SA(f)$ is the surface area of the function $f$ on domain, assuming that the domain is the same for all
nonparametric regression estimators and compact.  In the special case that $x$ is unidimensional, $SA(\hat{f}_{\hat{\lambda}})$
is just the arc length over its compact domain in $\mathbb{R}$ and when $x$ is two-dimensional, $SA(\hat{f}_{\hat{\lambda}})$ is the 
area of the surface defined by $\hat{f}_{\hat{\lambda}}$ over its compact  domain in $\mathbb{R}^2$.  
Essentially, we are ordering the random basis elements by how close they are to a full data function estimator
in terms of volume in $\mathbb{R}^{\dim{x}}$.
\section{Computed examples} 
\label{computedexamples}

In this section we apply our technique described in previous sections to one $\cal{M}$-complete data set and one $\cal{M}$-open data set.  The first is a `canned'
data set that is recognized to be difficult.  The second is a new data set on soil moisture graciously provided by Prof.
T. Franz, see \citet{trent}.

\subsection{Forest Fires data}
Consider the {\sf Forest Fires} data set publicly available from the UC Irvine Machine Learning Reposition. The sample size is $n=517$ and 
there are $8$ non-trivial explanatory variables related to the severity of a forest fire. The dependent variable is the burn area of the fire. 
Details and references can be found at {\texttt {http://archive.ics.uci.edu/ ml/datasets/Forest+Fires}}.
We regard the {\sf Forest Fires} data set as $\cal{M}$-complete because a forest fire is a chemical reaction with a lot of randomness that
cannot  be quantified well.  That is, there is so little about the process that is stable that it is unclear there is anything to estimate.
However, it is plausible that there is a model, necessarily highly complex, that might accurately encapsulate the behavior of forest fires under a 
variety of environmental conditions. Of course, such a model could be so complex that even though the data generator is   $\cal{M}$-complete 
it is nearly   $\cal{M}$-open.
Thus, generating predictions may be the most appropriate approach even they have a large variability.

For our analysis, we divide the data randomly into two subsets, one for training and one for validation. The training set contains $n_1=267$ data points and the 
validation set contains $n_2=250$ data points.   To generate a predictor, we assume the predictive analog of an additive model.
That is, we form eight univariate models, each using one of the explanatory variables, and then we stack the predictors they generate.
To generate each univariate model, we generated data-driven basis elements to use in a linear model. Each linear model has an associated
point predictor and these are weighted by their stacking coefficients.

We consider two classes of data-driven basis elements.  The first class is generated as discussed in Subsec. \ref{optchoice} using the Nadaraya-Watson
estimator found using the {\texttt{npreg()}} function in {\texttt{R}}.  The second class is generated using Gaussian process priors found using the
{\texttt{gausspr()}} function in {\texttt{R}}.  In both cases we used the default settings for the {\texttt{R}} functions
and we set the number of basis elements to find to be $J = 10$.   Unsurprisingly, $J_{opt}$ assumed values and the high end of its range,
8, 9, and 10, but most often 10.  Note that the value of $J_{opt}$ depends on which variable was being used and this process
is independent of the $m$ in the `sum to $m$' constraint.

We summarize our results in Fig.  \ref{fig:fig}.   The left hand panel in panel shows that when the Nadaraya-Watson estimator is used
to generate basis elements, then the optimal value of the constraint $m$ is $m_{opt}= .91$.  By contrast, the right hand panel
shows that when the Gaussian process prior is used, we find $m_{opt} = 1.28$.   
In this case generating basis elements by Nadaraya-Watson gives a better prediction, i.e. a much lower cumulative predictive error than the Gaussian process
prior method.  Since we used the same $J$ for the two estimators, we interpret our results in terms of a counterfactual:
the Nadaraya-Watson method leads to basis elements that generate subspaces that are closer to the subspace containing what the true function would be
if it existed.  Note also that both optimal values of $m$ are meaningfully different from one at least in terms of the predictive error they give.

\begin{figure}
\begin{subfigure}{.5\textwidth}
  \centering
  \includegraphics[width=.8\linewidth]{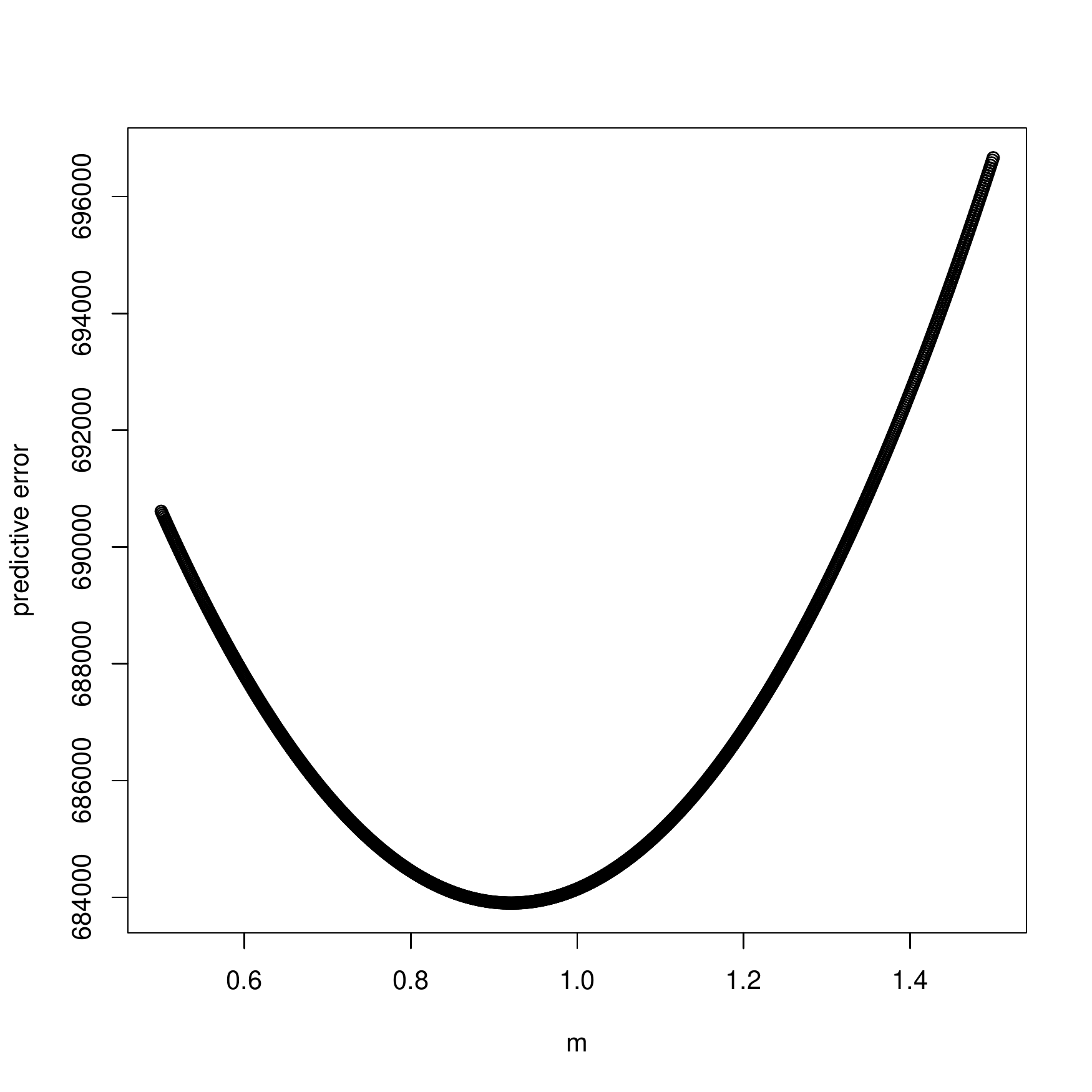}
  \caption{using {\texttt{npreg()}} }
  \label{fig:sfig1}
\end{subfigure}%
\begin{subfigure}{.5\textwidth}
  \centering
  \includegraphics[width=.8\linewidth]{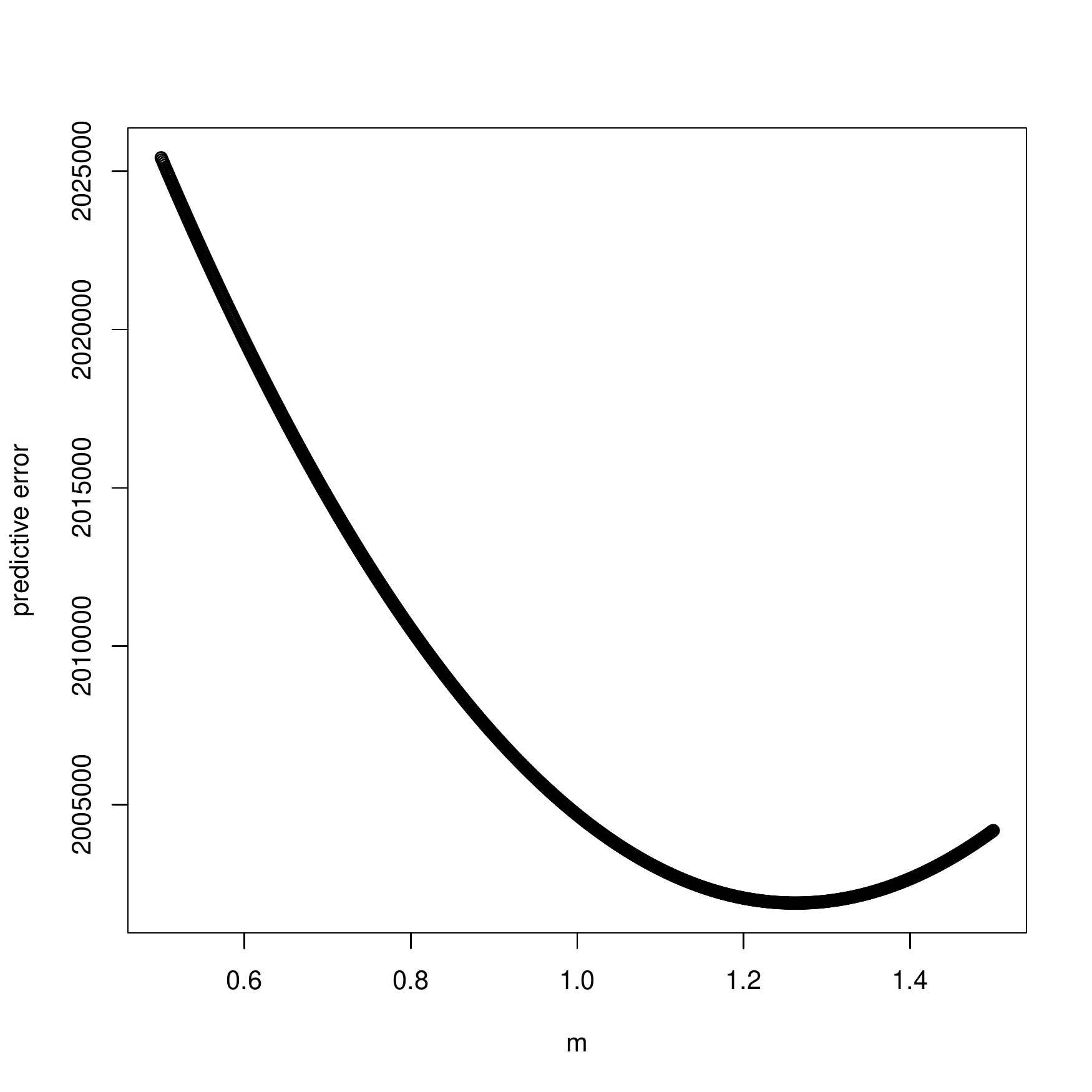}
  \caption{using {\texttt{gausspr()}}}
  \label{fig:sfig2}
\end{subfigure}
\caption{Plots of cumulative predictive error of stacking eight univariate predictors vs. $m$, the value of the constraint for
the {\sf Forest Fires} data.
Left:  Basis elements generated by Nadaraya-Watson.  Right:  Basis elements generated using Gaussian process priors
with a radial basis function kernel.}
\label{fig:fig}
\end{figure}
\subsection{Soil moisture data}\label{seceg}
As an  example that is $\cal{M}$-open, we consider the {\sf Soil Moisture} data set.   The response variable is an
interpolated form of the moisture in the topsoil.   There are six explanatory variables three for location
(two for location on a grid, one for elevation), two for soil electrical resistivity, and one for a standard `wetness index'
that is a function of elevation; see \citet{trent} for a detailed description.
The actual sample size is $18973$ but for computational convenience, we randomly selected $n=1000$
data points, dividing them into two sets of size 500, at random, for training and one for validation as before. 
We continued to set $J = 10$ and used Nadaraya-Watson and Gaussian process priors to generate basis elements.
This time we used a polynomial kernel in the Gaussian process prior because other kernels did not
permit convergence or gave the same result.
Again, we found $J_{opt}= 8, 9, 10$.  

The predictive error from stacking six univariate predictors for a range of
$m$ for basis elements generated using Nadaraya-Watson and Gaussian process priors are shown
in Fig. \ref{fig2:fig}.  It is seen that the predictive performance is nearly the same for both cases 
and, in particular, the optimal predictive errors are small and the
optimal constraint is nearly the same, $m_{opt} \approx 1$.  Comparing with the results for the
{\sf Forest Fires} data, we see that the better the predictive performance is, the closer to one
$m_{opt}$ is and that when the predictive performance is weaker values of $m_{opt}$ can
be either larger or smaller than one.  We suggest that had we chosen a smaller subset of the {\sf Soil Moisture} data
we would have found worse predictive performance and an $m_{opt}$ further from one.

\begin{figure}
\begin{subfigure}{.5\textwidth}
  \centering
  \includegraphics[width=.8\linewidth]{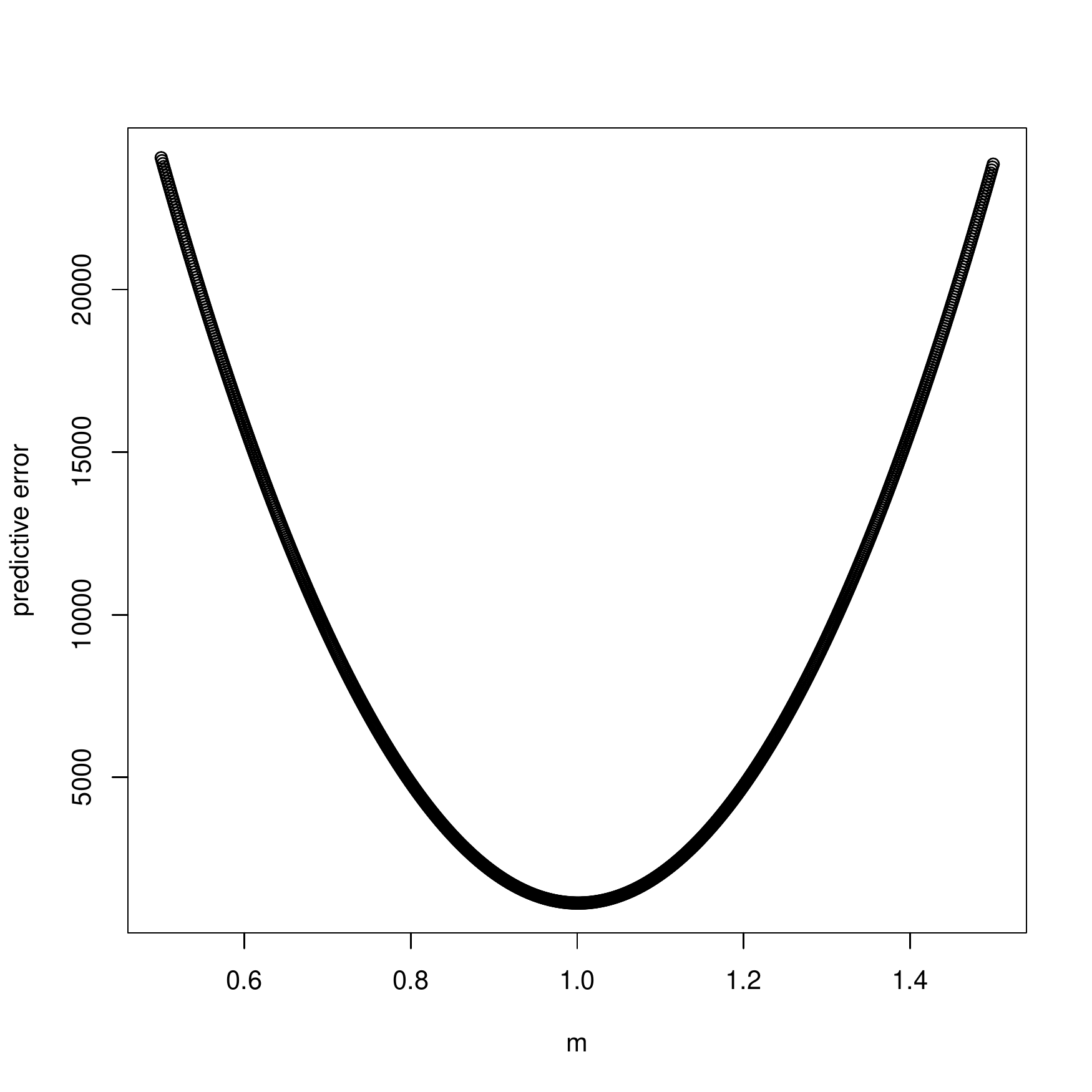}
  \caption{using {\texttt{npreg()}} }
  \label{fig2:sfig1}
\end{subfigure}%
\begin{subfigure}{.5\textwidth}
  \centering
  \includegraphics[width=.8\linewidth]{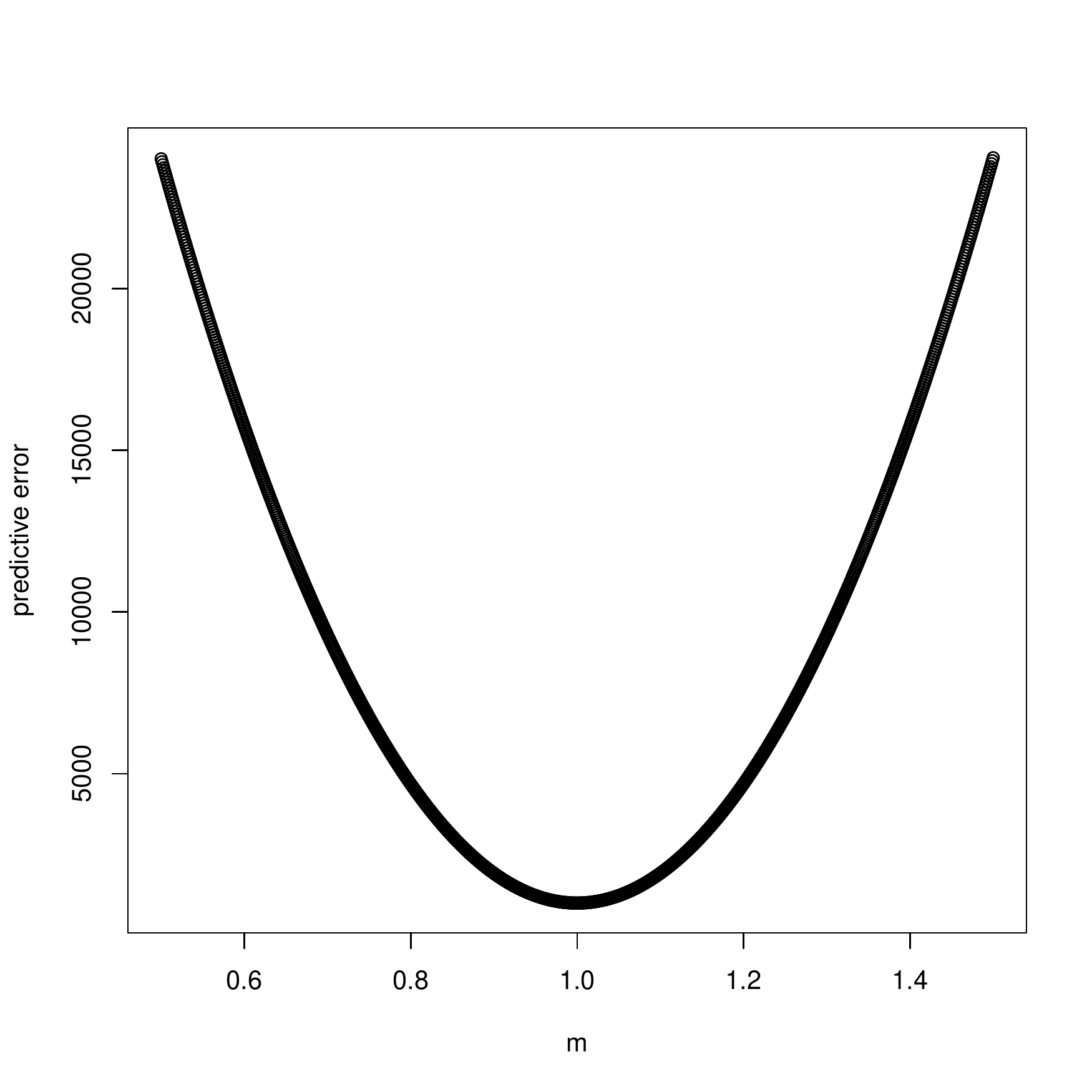}
  \caption{using {\texttt{gausspr()}}}
  \label{fig2:sfig2}
\end{subfigure}
\caption{ Plots of cumulative predictive error of stacking eight univariate predictors vs. $m$, the value of the constraint for
the {\sf Soil Moisture} data.
Left:  Basis elements generated by Nadaraya-Watson.  Right:  Basis elements generated using Gaussian process priors
with a polynomial kernel. }
\label{fig2:fig}
\end{figure}

\section{Discussion}
\label{Discussion}

Here we have formally established that leave-one-out cross-validation is 
asymptotically the optimal action in posterior risk for a variety of loss
functions.  We have used this to
justify the coefficients in a stacking predictor since they are based on a
cross-validation criterion.  Stacking is a model averaging
technique for prediction most effective when a true model is unavailable
or may not even exist.  We have investigated theoretically and
computationally the effect of different choices of constraints on
the coefficients of the stacking predictor and suggest that not imposing
any leads to the best result in the sense of minimizing predictive error.
In fact, our examples suggest that a `sum to one' constraint 
naturally emerges when the predictive error is small.
We comment that obvious extensions of our technique of proof show
that leave-$k$-out cross-validation by also be regarded as Bayes actions.

When the concept of a true model is problematic, it is natural to fall back on
predictive methods.  Indeed, it is possible that seeking a good predictor may
be more useful that modeling when the model is very complex.  For instance, if 
no simplification of the true model can
be readily identified a model average predictor may give better performance
in a mean squared error sense.  This seems to be the case for our
two examples here.

Finally, we recall that \citet{stone}  showed that the Akaike information criterion
(AIC) is asymptotically equivalent to leave-one-out cross-validation and that \citet{shao} shows 
these further asymptotically equivalent to the Mallows' $C_p$ criterion, the generalized
cross-validation, and the `$GIC_2$' criterion.  The implication from our main theorem
here is that all of these methods can also be regarded as asymptotically 
Bayes optimal.

\appendix

\section*{Appendix} 
\subsection*{Proof of the example in  Section 1}
Consider the two models $M_1: Y = x_1\beta_1 + \epsilon$ and $M_2: Y = x_2\beta_2 + \epsilon$ where
the explanatory variables are orthogonal i.e., $x_1' x_2 = 0$, $E(\epsilon) = 0$, and $Var(\epsilon) = \sigma^2$.  If we stack these two 
models with the sum to one constraint, by Corollary \ref{cor1}, we have
\begin{eqnarray}\label{A38}
\begin{aligned}
	\hat{w}_1 &\propto \sum_{i=1}^n \hat{e}_{i,2}^2 - \hat{e}_1' \hat{e}_2 = \sum_{i=1}^n \frac{e_{i,2}^2}{(1-h_{ii,2})^2} - \hat{e}_1' \hat{e}_2, \\
	\hat{w}_2 &\propto \sum_{i=1}^n \hat{e}_{i,1}^2 - \hat{e}_1' \hat{e}_2 = \sum_{i=1}^n \frac{e_{i,1}^2}{(1-h_{ii,1})^2} - \hat{e}_1' \hat{e}_2, 
\end{aligned}
\end{eqnarray} 	 
where $e_{i,j}$ and $h_{ii,j}$ are the ordinary residual and the leverage for case $i$ under model $M_j$, respectively.  

Since $h_{ii,j} = x_{ij}^2/\sum_{i=1}^n x_{ij}^2$, if $h_{ii,j}\rightarrow 0$ as $n\rightarrow\infty$, then  (\ref{A38}) becomes
\begin{eqnarray*}
\begin{aligned}
	\hat{w}_1 &\propto  \sum_{i=1}^n {e_{i,2}^2} - \hat{e}_1' \hat{e}_2 = (n-1)\sigma^2 - \hat{e}_1' \hat{e}_2 , \\
	\hat{w}_2 &\propto  \sum_{i=1}^n {e_{i,1}^2} - \hat{e}_1' \hat{e}_2 = (n-1)\sigma^2 - \hat{e}_1' \hat{e}_2. 
\end{aligned}
\end{eqnarray*} 	 
Combining with the sum to one constraint, this yields $\hat{w}_1 = \hat{w}_2 = 1/2$.

Now, if we stack $M_1$ and $M_2$ with a sum to two constraint, then from Theorem \ref{thm3.2n}  and similar arguments 
as above we have
\begin{eqnarray*}
\begin{aligned}
	\hat{w}_1 &\propto   (n-1)\sigma^2 + \sum_{i=1}^n \hat{y}_{2,(-i)}^2 - \hat{e}_1' \hat{e}_2 -  \sum_{i=1}^n \hat{y}_{1,(-i)}\hat{y}_{2,(-i)}, \\
	\hat{w}_2 &\propto  (n-1)\sigma^2 + \sum_{i=1}^n \hat{y}_{1,(-i)}^2 - \hat{e}_1' \hat{e}_2 -  \sum_{i=1}^n \hat{y}_{1,(-i)}\hat{y}_{2,(-i)}. 
\end{aligned}
\end{eqnarray*} 	 
So, if $\sum_{i=1}^n \hat{y}_{1,(-i)}^2 = \sum_{i=1}^n \hat{y}_{2,(-i)}^2$ then combining with the sum to two constraint
we get the weights now $\hat{w}_1 = \hat{w}_2 = 1$.

\subsection*{Proof of Theorem \ref{thm4.2}}
Without loss of generality, assume $\cal{M}$ and $\cal{M}^\prime$ are bases of $\langle {\cal{M}} \rangle = \langle  {\cal{M}}^\prime \rangle$
and hence $J=J'$.

For the  basis $\cal{M}$,  we have the decomposition
\begin{eqnarray}\label{36}
	y = y_1+y_2,
\end{eqnarray} 	 
where $y_1 = \sum_{j=1}^J  \alpha_j  f_j$, $y_2 = \sum_{j>J}  \langle y, e_j \rangle e_j$, and $\{e_1, e_2, \cdots\}$ is an orthonormal basis for ${\cal{H}}$.
Then,  Theorem \ref{thm3.3} gives
\begin{eqnarray*}
\begin{aligned}
	Q_{\min}^{\cal{M}} &= \left\| (y(x_1),\cdots,y(x_n))' - \sum_{j=1}^J  \hat{w}_j (f_j(x_1),\cdots,f_j(x_n))^\prime\right\|^2\\
	&= \left\| \left(\sum_{j=1}^J  \alpha_j  f_j (x_1)+ \sum_{j>J}  \langle y, e_j \rangle e_j(x_1),  \right.\right.\\
	&\quad\quad\quad\quad\quad\quad\quad\left.\cdots, \sum_{j=1}^J  \alpha_j  f_j (x_n)+ \sum_{j>J}  \langle y, e_j \rangle e_j(x_n)\right)^\prime\\
	 &\quad\quad- \left.\sum_{j=1}^J  \hat{w}_j (f_j(x_1),\cdots,f_j(x_n))' \right\|^2\\  
	&= \left\|  \sum_{j=1}^J (\alpha_j-\hat{w}_j)( f_j(x_1),\cdots,f_j(x_n) )'  + \sum_{j>J}  \langle y, e_j \rangle \left( e_j(x_1),\cdots,  e_j(x_n)\right)'  \right\|^2.
\end{aligned}
\end{eqnarray*} 	 
Since $f_j\in \langle\{ e_1,\cdots, e_J \}\rangle$ for $j=1,\cdots,J$ and $e_j\in \langle\{ e_1,\cdots, e_J \}\rangle^\perp$ for $j>J$, then
\begin{eqnarray}\label{37}
\begin{aligned}
	Q_{\min}^{\cal{M}} 
	&= \left\|  \sum_{j=1}^J (\alpha_j-\hat{w}_j)( f_j(x_1),\cdots,f_j(x_n) )'  \right\|^2 \\
	&	\quad\quad\quad+ \left\|  \sum_{j>J}  \langle y, e_j \rangle \left( e_j(x_1),\cdots,  e_j(x_n)\right)'  \right\|^2\\
	&= \left\|  \sum_{j=1}^J (\alpha_j-\hat{w}_j)( f_j(x_1),\cdots,f_j(x_n) )'  \right\|^2 + \| y_2 \|^2.
\end{aligned}
\end{eqnarray} 	 
Now, from (\ref{36}),
\begin{eqnarray*}\label{4.5n}
\begin{aligned}
	\langle y,  f_l \rangle  &= \sum_{j=1}^J \alpha_j \langle f_j,  f_l \rangle +  \sum_{j=J+1}^n \langle y, e_j \rangle  \langle e_j, f_l\rangle\\
	&= \sum_{j=1}^J \alpha_j \langle f_j,  f_l \rangle,
\end{aligned}
\end{eqnarray*} 
for $l=1,\cdots,J$. Therefore,
\begin{eqnarray*}
	\alpha  = T ^{-1} c  = \hat{w},
\end{eqnarray*} 
where $T$ and $c$ are given in Theorem \ref{thm3.3}, and hence (\ref{37}) yields
\begin{eqnarray*}\label{49n}
\begin{aligned}
	Q_{\min}^{\cal{M}}  =  \| y_2 \|^2.
\end{aligned}
\end{eqnarray*} 
This result does not depend on ${\cal{M}}$; therefore, $Q_{\min}^{\cal{M}}=Q_{\min}^{\cal{M}'}$.

\subsection*{Proof of Theorem \ref{thm4.4n}}
Without loss of generality, assume $\cal{M}$ is orthonormal. Then, for  $\cal{M}$, as in the proof of Theorem \ref{thm4.2}, we have
\begin{eqnarray}\label{4.51n}
\begin{aligned}
	Q_{\min}^{\cal{M}}  =  \| y_2 \|^2 = \sum_{j>J} \langle y, e_j \rangle^2.
\end{aligned}
\end{eqnarray} 	 
For  $\cal{N}$, we now have the decomposition 
\begin{eqnarray*}\label{4.52n}
	y = \sum_{j=1}^{J-1} \langle y, f_j \rangle f_j +  \sum_{j\ge J} \langle y, e_j \rangle e_j.
\end{eqnarray*} 
Then, from Theorem \ref{thm3.3}, 
\begin{eqnarray}\label{4.53n}
\begin{aligned}
	Q_{\min}^{\cal{N}} &= \left\| y - \sum_{j=1}^{J-1}  \hat{w}_j f_j \right\|^2\\
	&= \left\| \sum_{j=1}^{J-1} \langle y, f_j \rangle f_j +  \sum_{j\ge J} \langle y, e_j \rangle e_j - \sum_{j=1}^{J-1}  \hat{w}_j f_j \right\|^2\\
	&= \left\| \sum_{j=1}^{J-1} ( \langle y, f_j \rangle  - \hat{w}_j) f_j   +  \sum_{j\ge J}^n \langle y, e_j \rangle e_j \right\|^2\\
	&= \left\| \sum_{j=1}^{J-1} ( \langle y, f_j \rangle  - \hat{w}_j) f_j   \right\|^2 +  \sum_{j\ge J} \langle y, e_j \rangle ^2.
\end{aligned}
\end{eqnarray} 	 

The desired inequality in Theorem \ref{thm4.4n} is obtained from (\ref{4.51n}) and (\ref{4.53n}).

\bibliography{BayesSTK-arxiv}

\end{document}